\documentclass[12pt]{amsart}
\usepackage{amssymb,amsmath,amscd,graphicx,amsthm,enumerate,latexsym, subfig}
\usepackage[all]{xy} 
\usepackage{ytableau}

\newtheorem{theorem}{Theorem}[section]
\newtheorem{corollary}[theorem]{Corollary}
\newtheorem{lemma}[theorem]{Lemma}
\newtheorem{proposition}[theorem]{Proposition}
\newtheorem{example}[theorem]{Example}
\newtheorem{definition}[theorem]{Definition}
\newtheorem{remark}[theorem]{Remark}

\numberwithin{equation}{section}
\newcommand{\op}{\operatorname{op}}
\newcommand{\sgn}{\operatorname{sgn}}

\newcommand{\Pp}{\mathcal{P}}

\newcommand{\Z}{\mathbb{Z}}
\newcommand{\rk}{\textrm{rank }}
\newcommand{\C}{k}

\newcommand{\Rep}{\textrm{Rep}}
\newcommand{\N}{\mathbb{N}}

\newcommand{\gr}{\operatorname{gr}}

\newcommand{\SI}{\textrm{SI}}
\newcommand{\PEG}{\operatorname{PEG}}

\newcommand{\GL}{\textrm{GL}}

\newcommand{\SL}{\textrm{SL}}

\newcommand{\Compl}{\operatorname{Com}}

\newcommand{\Hom}{\textrm{Hom}}

\title[Semi-Invariants for Gentle String Algebras]{Semi-Invariants for Gentle String Algebras}
\author{Andrew T. Carroll, Jerzy Weyman}
\address{Department of Mathematics,
		Northeastern University,
		Boston, MA 02115}
\thanks{The first author was supported by NSF grant number DMS-0801220.  The second author was partially supported by NSF grant DMS-0901185.}
\email{carroll.a@husky.neu.edu, j.weyman@neu.edu}
\begin{document}


\begin{abstract}
In this article we give an algorithm for determining the generators and relations for the rings of semi-invariant functions on irreducible components of $\Rep_A(\beta)$ when $A$ is a (acyclic) gentle string algebra and $\beta$ is a dimension vector.  These rings of semi-invariants turn out to be semigroup rings to which we can associate a so-called matching graph.  Under this association, generators for the semigroup can be seen by certain walks on this graph, and relations are given by certain configurations in the graph.  This allows us to determine degree bounds for the generators and relations of these rings.  We show further that these bounds also hold for acyclic string algebras in general.  
\end{abstract}
\maketitle
\section{Introduction}
Gentle string algebras are a generally well-behaved class of algebras, which are special cases of (special) biserial algebras.  They are tame, but exhibit so-called non-polynomial growth.  Gel'fand and Ponomarev \cite{GP} described the indecomposable representations of the Lorentz group, the generalities of which would lead to the definition of biserial algebras.  For the case of string algebras, Butler and Ringel \cite{BR} have described all indecomposable modules, determined the irreducible morphisms, and the possible components that can arise in the Auslander-Reiten quiver.  More recently, string algebras have appeared in connection to cluster algebras, namely from triangulations of unpunctured surfaces with marked points \cite{ABCP}.  The motivation and much of the setup in this article are generalizations of those exhibited by Kra\'{s}kiewicz and Weyman in \cite{WK}.  The main theorem of this article is, in fact, a result in the direction of a conjecture posed by the aforementioned authors concerning so-called matching semigroups.  

Fix an algebraically closed field $\C$.  Suppose that $\C Q/I$ is a gentle string algebra, and $Q$ is a quiver without oriented cycles.  We will show that the components of $\Rep_{\C Q/I}(\beta)$ are parameterized by maps $r:Q_1\rightarrow \N$ satisfying certain properties with respect to the dimension vector $\beta$.  These maps will be called \emph{rank sequences}, and allow us to state the main theorem of the paper:
\begin{theorem}
$\C[\Rep_{\C Q/I}(\beta)_r]^{\SL(\beta)}$ is a semigroup ring with generators in degree at most \[
\sum\limits_{a\in Q_1} 2 \binom{r(a)+1}{2}.\]  Furthermore, the relations in this ring occur in degree bounded by 
\[
	\sum\limits_{a\in Q_1} 8 \binom{r(a)+1}{2} .
	\]
\end{theorem}
The key remark concerning gentle string algebras is that any string algebra $\C Q/I$ for which $Q$ is acyclic is a quotient of a gentle string algebra with the same underlying quiver.  This implies that irreducible components in $\Rep_{\C Q/I}(\beta)$ are entirely contained in those of $\Rep_{\C Q/I'}(\beta)$.  In particular, the ring of semi-invariants for $\Rep_{\C Q/I}(\beta)$ is a quotient ring of those rings herein calculated for $\Rep_{\C Q/I'}(\beta)$, so bounds in the latter ring hold for the former.  

The article is organized as follows.  In section \ref{sec:preliminaries}, we give some basic definitions regarding gentle string algebras.  Section \ref{sec:complexes} introduces the varieties of complexes, as studied in \cite{DCS}, and exhibits the irreducible components.  In section \ref{sec:TheCoordinateRing}, we give an explicit basis for the coordinate rings as products of minors, thus allowing us to show a $\GL(\beta)$-decomposition of these rings.  This is primarily a collection of results from \cite{DCS} (see also \cite{DRS}).  Section \ref{sec:coordinateringofrep} extends these results to the spaces $\Rep_{\C Q/I}(\beta)$ by viewing these spaces as products of varieties of complexes.  In this way, we will be able to determine the summands of the decomposition of the coordinate ring and show that the subring of semi-invariant functions is a semigroup ring.  Section \ref{sec:COMBINATORICS} exploits the calculations necessary to determine the generators of the aforementioned semigroup, and \ref{sec:semigroupU} gives the combinatorial framework for determining the explicit generators and relations in this semigroup.  This allows us to determine degree bounds for the generators and relations of the generating semi-invariants in section \ref{sec:corollary}.


\section{Preliminaries}\label{sec:preliminaries}

Fix an algebraically closed field $\C$.  A \emph{quiver} $Q=(Q_0, Q_1)$ is a directed graph with $Q_0$ the set of vertices, and $Q_1$ the set of arrows.  Such a quiver comes equipped with two maps $t,h:Q_1 \rightarrow Q_0$ with $ta:=t(a)$ the tail of the arrow $a$, and $ha:=h(a)$ the head of the arrow.  As usual, the path algebra $\C Q$ of the quiver $Q$ is the algebra with basis given by all paths $p$ on $Q$ (together with the length-zero paths $e_x$ for $x\in Q_0$) where the multiplication is given by concatenation of paths, written as composition of maps.  We will often use the symbol $[n]$ for the interval $\{1,2,\dotsc, n\}$ for the sake of compact notation. We now recall the definition of gentle string algebras following Assem and Skowro\'{n}ski in \cite{AS}.

\begin{definition} A finite-dimensional $\C$-algebra $A$ is called a \emph{gentle string algebra} if it admits a presentation $\C Q/I$ satisfying the following properties:
\begin{itemize}
\item[a.] for each vertex $x\in Q_0$, the number of arrows $a$ with $ha=x$ is bounded by 2, and the number of arrows $b$ with $tb=x$ is bounded by 2;
\item[b.] for each arrow $b$ there is at most one arrow $a$ with $ta=hb$ (resp. at most one arrow $c$ with $hc=tb$) such that $ab\notin I$ (resp. $bc\notin I$);
\item[c.] for each arrow $b$ there is at most one arrow $a$ (resp. at most one arrow $c$), as above such that $ab\in I$ (resp. $bc \in I$). 
\item[d.] $I$ is generated by paths of length 2.
\end{itemize}
If $\C Q/I$ satisfies only (a) and (b), then it is said to be a string algebra.
\end{definition}

\begin{definition}
A \emph{coloring} of a quiver $Q$ is a map $c:Q_1 \rightarrow S$, where $S$ is some finite set (whose elements we call colors), such that $c^{-1}(s)$ is a directed path for every $s\in S$.  
\end{definition}

\begin{example}
Let $Q$ be the quiver 
\begin{align*}
\xymatrix{
1 \ar[r]^{a_1} & 2 \ar[r]^{a_2}\ar[dr]_{b} &3 \ar[r]^{a_3}& 5\\ && 4}
\end{align*}
The map $c:Q_1 \rightarrow \{1,2\}$ with $c(a_1)=c(b)=1$ and $c(a_2)=c(a_3)=2$ is a coloring, since $c^{-1}(1)=\{b, a_1\}$ forms a path and $c^{-1}(2)=\{a_3, a_2\}$ as well; however, the map to the same set with $c(a_1)=c(a_3)=1$ and $c(a_2)=c(b)=2$ is not a coloring, since $c^{-1}(2)=\{b, a_2\}$, which cannot be made into a path.  
\end{example}

\begin{remark}
A coloring is simply a partition of the arrows into a disjoint union of paths.  We will often describe such a coloring pictorially by depicting arrows of different colors by different arrow types.
\end{remark}

\begin{definition}Fix a quiver $Q$ and a coloring $c$ of the quiver, define by $I_c$ the ideal generated by all monochromatic paths of length 2, i.e.,  $$I_c = <ba \mid c(a)=c(b),\textrm{ and } ha=tb>.$$  The algebra $kQ/I_c$ is called a \emph{colored algebra} (we will sometimes simply say that $(Q, c)$ is a colored algebra when the field is understood).  
\end{definition}

\begin{proposition}
If $\C Q/I$ is a gentle string algebra such that $Q$ has neither loops nor oriented cycles, then there is a coloring $c$ on $Q$ such that $I_c = I$.  
\end{proposition}
\begin{proof}
Let $S$ be a set with elements in bijection with the set of arrows $a\in Q_1$ such that there is no $b\in Q_1$ with $hb=ta$ and $ab\in I$.  Let $s_a\in S$ be the element corresponding to such an $a\in Q_1$ under this bijection.  For each element $s_a\in S$, let $p(a)=p_{l(a)}(a) \dotsc p_1(a)$ be the longest path with $p_1(a)=a$ and $p_{i+1}(a)p_i(a)\in I$.  Notice first that the length is bounded since $Q$ is acyclic.  Additionally, this path is unique and well-defined since for each arrow $p_i(a)$ there is at most one arrow $p_{i+1}(a)$ such that $p_{i+1}(a)p_i(a)\in I$.  
Take $c:Q_1\rightarrow S$ to be the map with $c(p_i(a))=s_aa$ for each $i=1,\dotsc, l(a)$.  By definition of the gentle string algebras, for each $b$ there is at most one arrow $a$ with $ha=tb$ and $ba\notin I$.  Therefore, since $I$ is generated by paths of length 2, so $I_c=I$.  
\end{proof}

\begin{example} \label{examples-CSA} Consider the following examples of gentle string algebras together with their colorings $c$.
\begin{itemize}
\item[i.] Let $Q$ be any orientation of $A_n$, and let $c$ be any coloring of $Q$.  Then $kQ/I_c$ is a colored string algebra.  
\item[ii.] Let $A(n)$ be the quiver on $n+1$ vertices $Q_0 = \{1, \dotsc, n+1\}$ with arrows $a_i, b_i: i\rightarrow i+1$ for $i=1, \dotsc, n$, and the coloring $C=\{b_n b_{n-1} \dotsc b_1;\ a_n a_{n-1} \dotsc a_1\}$.  Then $kQ/I_C$ is a colored string algebra.  The general modules and rings of semi-invariants for this class were studied by Kra\'{s}kiewicz and Weyman in \cite{WK}.  
\item[iii.] The following will be a running example for this paper: 
\begin{align*}
\xymatrix{
			&			& 3 \ar@{->}[dr]^{c_1} && \\
1 \ar@<-.7ex>@{~>}[r]_{a_1} \ar@<.7ex>@{..>}[r]^{b_1} & 2 \ar@{~>}[ur]^{a_2} \ar@{..>}[rr]_{b_2} && 4 \ar@<.7ex>@{..>}[r]^{b_3} \ar@<-.7ex>@{->}[r]_{c_2}& 5 }
\end{align*}
Then the ideal $I_c$ is generated by $\{a_2a_1, b_2b_1, b_3b_2, c_2c_1\}$.
\end{itemize}
\end{example}
\begin{proposition}
Let $\C Q/I$ be a string algebra such that the underlying quiver $Q$ has neither loops nor oriented cycles.  Then there is a coloring $c$ of $Q$ and an algebra epimorphism $p:kQ/I_c \rightarrow kQ/I$ such that $kQ/I_c$ is a gentle string algebra.
\end{proposition}
\begin{proof}
We will work by induction on the number of arrows in $Q$.  If there is only one arrow, then the proposition is clear, taking $c(a)=s$.  Suppose the proposition holds for all acyclic quivers $Q'$ with $\lvert Q'_1\rvert < \lvert Q_1 \rvert$, and let $x$ be a source in $Q$ (such a vertex exists because $Q$ is acyclic).  Pick $q_1$ an arrow commencing at $x$.  We define a path $q$ by recursively defining its initial subpaths $q(i)=q_i \cdot \dotsc \cdot q_1$ as follows:
\begin{itemize}
\item[(1)] if there is a unique $b\in Q_1$ such that $tb=hq_i$ and $bq_i \in I$, then $q_{i+1}:=b$;
\item[(2)] if there are two arrows $b_1, b_2\in Q_1$ with $tb_1=tb_2=hq_i$ and $b_1q_i, b_2q_i \in I$, then we have the following possibilities:
\begin{itemize}
\item[(2a)] if there is no arrow $a'$ with $h(q_i)=h(a')$, then pick $q_{i+1}:=b_1$ (either arrow would suffice);
\item[(2b)] if, on the other hand, there is an arrow $a'$ with $h(q_i)=h(a')$ and $b_2a'\in I$, (say), then take $q_{i+1}:=b_1$.
\end{itemize}
\item[(3)] Finally, if there is no arrow $b$ with $0\neq bq_i \in I$, then take $q=q(i)$.  
\end{itemize}
Now consider quiver $Q'=(Q_0', Q_1')$ with $Q_0'=Q_0$ and $Q_1'=Q_1\setminus \{a \in q\}$, and the ideal $I'\subset kQ'$ the ideal induced by $I$ by setting $a=0$ when $a\in q$.  Notice that $kQ'/I'$ is still a string algebra, so by induction there is a coloring $c':Q_1'\rightarrow \{2,\dotsc, s\}$ and an algebra epimorphism $p': kQ'/I_{c'} \rightarrow kQ'/I'$ where $kQ'/I_{c'}$ is a colored string algebra.  Take $c: Q_1 \rightarrow \{1,\dotsc, s\}$ the map with $c(a) = c'(a)$ if $a\in Q_1'$ and $c(a)=1$ if $a\in q$.  Then $kQ/I$ is the quotient of $kQ/I_c$ for this coloring by additional relations..  
\end{proof}

\begin{example} \label{example-coloring}
Consider the quiver 
\xymatrix@1@R=.3ex@C=3ex{1 \ar[r]^{a_1} & {2} \ar[dr]^{a_2}& &   {3} \\
		&			&{4} \ar[dr]_{b_3} \ar[ur]^{a_3}&	\\
{5}	\ar[r]_{b_1}	&	{6} \ar[ur]_{b_2} & 		&	{7}}.
subject to $I=<a_3a_2,\ a_2a_1,\ b_3b_2,\ b_3a_2>$  Starting with vertex (1), and $q_1=a_1$, we see that one must have $q=a_3a_2a_1$.  So 
\begin{align*}Q'&=\xymatrix@R=1ex{{1}  & {2} & &   {3} \\
		&			&{4} \ar[dr]_{b_3} &	\\
{5}	\ar[r]_{b_1}	&	{6} \ar[ur]_{b_2} & 		&	{7}}\end{align*}
and $I'=b_3b_2$.  This is indeed a colored string algebra, when colored by $c'(b_3)=c'(b_2)=3$, $c'(b_1)=2$, so we would like to extend this coloring to one on $Q$ itself by following the proof.  This gives $c(b_3)=c(b_2)=3$, $c(b_1)=2$, and $c(a_i)=1$.  It is clear that $kQ/I_c$ projects onto $kQ/I$ with kernel $<b_3a_2>$.  \end{example}

\subsection{Representation Spaces}
Recall that a dimension vector for a quiver $Q$ is a vector $\beta\in \N^{Q_0}$.  Suppose that $I$ is generated by paths, say $p_1,\dotsc, p_l$.  Then the space of representations of $\C Q/I$ of dimension $\beta$ is
\[
\Rep_{\C Q/I}(\beta):=\{V \in \Rep_{kQ/I}\mid \dim V_x = \beta_x,\ x\in Q_0\}.
\]  Fixing vector spaces $V_x$ of dimensions $\beta_x$ for $x\in Q_0$, one can view this space as 
$$\Rep_{\C Q/I}(\beta)=\left\{(V_a)_{a\in Q_1} \in \prod\limits_{a\in Q_1} \Hom_k(V_{ta}, V_{ha}) \mid V(p_i) = 0 \textrm{ for all } i=1,\dotsc, l\right\},$$ where $V(p_i)$ is the composition of the maps corresponding to the arrows in the path $p_i$.  Notice that the algebraic group $\GL(\beta)=\prod_{x\in Q_0} \GL(\beta_x)$ acts linearly on $\Rep_{\C Q/I}(\beta)$, and the orbits correspond to isoclasses of $\C Q/I$ modules.

In the context of this article, suppose that $kQ/I_c$ is a gentle string algebra with $c: Q_1 \mapsto \{1,\dotsc, s\}$, and let $C=\{c_1,\dotsc, c_s\}$ the set of colored paths associated with $c$.  Suppose that $c_i=a_{n_i}^{(i)}\dotsc a_{2}^{(i)}a_1^{(i)}$ for each $i$.  Then for fixed $\beta$, we have $$\Rep_{\C Q/I_c}(\beta) = \prod\limits_{i=1}^s \left\{(V_{a^{(i)}_j})_{j=1,\dotsc, n_i} \in \prod\limits_{j=1}^{n_i} \Hom_\C(V_{ta_j^{(i)}}, V_{ha_j^{(i)}})\mid V_{a^{(i)}_{j+1}}\circ V_{a^{(i)}_j}=0\right\},$$ i.e., as a variety $\Rep_{(Q, c)}(\beta)$ is the product of the well-known varieties of complexes.  We remark that this result holds in more generality than gentle string algebras: any algebra with presentation $\C Q/I$ that admits a coloring $c$ such that $I=I_c$ has the property that $\Rep_{\C Q/I}(\beta)$ is the product of varieties of complexes.  

\subsection{Semi-Invariants}
Though many of these definitions hold for $G$-varieties in general, we will state them only for representation spaces of a quiver.
\begin{definition}
Suppose that $\C Q/I$ is a quiver with relations, and let $\Rep_{\C Q/I}(\beta)_i$ be an irreducible component of the representation space.  Then 
\[
	\SI_{\C Q/I}(\beta, i):= \C[\Rep_{\C Q/I}(\beta)_i]^{\SL(\beta)},
	\]
where $\SL(\beta):=\prod\limits_{x\in Q_0} \SL(\beta_x)$, is called the \emph{ring of semi-invariant} functions on $\Rep_{\C Q/I}(\beta)_i$.  The aforementioned ring has a decomposition into \emph{weight spaces} 
\[
	\SI_{\C Q/I}(\beta, i)=\bigoplus_{\chi\in \operatorname{char} \GL(\beta)} \SI_{\C Q/I}(\beta, i)_\chi\] where 
\[
	\SI_{\C Q/I}(\beta, i)_\chi:= \{f\in \C[\Rep_{\C Q/I}(\beta)_i] \mid g.f = \chi(g)\cdot f \textrm{ for all } g\in \GL(\beta)
	\]
\end{definition}
We will calculate the ring of semi-invariants, however, the weight space decomposition remains unknown.


\section{The Variety of Complexes}\label{sec:complexes}
We recall the definition of the varieties of complexes as given by DeConcini and Strickland in \cite{DCS}.  We describe their irreducible components, and a decomposition of their coordinate rings by means of Schur modules.  

Fix a non-negative integer $n$, and an element $\beta \in \N^{n+1}$.  Let $\{V_i\}_{i=1, \dotsc, n+1}$ be a collection of $k$-vector spaces with $\dim_k V_i = \beta_i$.  Define the variety of complexes by $$\Compl_{n+1}(\beta):= \left\{(A_i)_{i\in[n]} \in \bigoplus\limits_{i\in[n]} \Hom(V_i, V_{i+1}) \mid A_{i+1}\cdot A_i = 0\quad i\in [n]\right\}.$$  We simply write $\Compl(\beta)$ if the length $n$ is understood.  Notice that the algebraic group $\GL(\beta)=:\prod\limits_{i\in [n+1]}\GL(\beta_i)$ acts by simultaneous change of basis on $\Compl_{n+1}(\beta)$.  We will now describe the orbits of this action.
\begin{definition} A sequence of natural numbers $r=(0:=r_0, r_1, \dotsc, r_n, r_{n+1}:= 0)$ is called a \emph{rank sequence for $\beta$} if $r_{i-1} + r_i \leq \beta_i$ for $i\in [n+1].$  Furthermore, we define a partial order on these rank sequences by the condition $r\preceq s$ if and only if $r_i \leq s_i\ $ for each $i\in [n]$. \end{definition}
\begin{definition}
For a fixed rank sequence $r$ for $\beta$, we denote by $\Compl(\beta, r)$ (resp. $\Compl^\circ(\beta, r)$) the subsets of elements $(A_i)_{i\in [n]}$ in $\Compl(\beta)$ such that $\rk A_i \leq r_i$ (resp. $\rk A_i < r_i$) for each $i=1,\dotsc, n$.  
\end{definition}
Notice that if $\underline{A}\in \Compl(\beta)$ then the sequence $r(\underline{A})$ with $\rk A_i \leq r(\underline{A})_i$ is a rank sequence.  

\begin{proposition}\label{rank-sequence-proposition}
If $r$ is a rank sequence for $\beta$, then $\Compl^\circ(\beta, r)$ is an orbit in $\Compl(\beta)$ under $\GL(\beta)$ and $\Compl(\beta, r)$ is its closure.  
\end{proposition}
\begin{proof}
It is instructive to consider this proof by understanding representations of the quiver $$\mathcal{A}_{n+1}^{eq}=1 \xrightarrow{a_1} 2 \xrightarrow{a_2} \dotsc \xrightarrow{a_n} n+1$$ of dimension vector $\beta$.  For $1\leq i \leq j \leq n+1$, let $E_{i, j}$ be the representation of $\mathcal{A}_{n+1}^{eq}$ with $$E_{i, j}(l) = \begin{cases} k & \textrm{if } i\leq l \leq j \\ 0 & \textrm{ otherwise},\end{cases}$$ together with maps 
\begin{align*}
E_{i,j}(a_l)&: E_{i,j}(l) \rightarrow E_{i,j}(l+1)\\
E_{i,j}(a_l)&: 1 \mapsto 1 \quad \textrm{ if } i\leq l <j\\ E_{i,j}(a_l) &= 0 \qquad \textrm{ otherwise}.\end{align*}
It is well known that representations $E_{i, j}$ for $i\leq j$ give a complete list of indecomposable representations of $\mathcal{A}_{n+1}^{eq}$, so any representation is a direct sum of these.  It is furthermore clear that $\Compl_{n+1}(\beta) \subset \Rep(\mathcal{A}_{n+1}^{eq}, \beta)$, and so any element in $\Compl_{n+1}(\beta)$ is a direct sum of representations $E_{i, j}$ which satisfy the relation $E_{i, j}(a_{l+1}) \circ E_{i, j}(a_l) = 0$ for $l\in [n]$.  This leaves only the representations $E_{i, i+1}$ for $i\in [n]$, and $E_{i,i}$ for $i \in [n+1]$.  Therefore, any element $M$ in $\Compl_{n+1}(\beta)$ is contained in the $\GL(\beta)$-orbit with $$M(\underline{t}, \underline{s})=\bigoplus\limits_{i\in [n]} E_{i, i+1}^{t_i} \oplus \bigoplus\limits_{i\in [n+1]} E_{i,i}^{s_i}$$ for some $\underline{t}\in \N^{n},\ \underline{s} \in \N^{n+1}$ satisfying the condition that $t_i + t_{i-1}+ s_i = \beta_i$ for $i\in [n+1]$.  Notice that $\rk(M(\underline{t}, \underline{s}))_i = t_i$, so $\Compl^\circ(\beta, r)$ is a single orbit in $\Compl(\beta)$.

We now show that $\Compl(\beta, r)$ is the closure of $\Compl^\circ(\beta, r)$.  Notice that $\Compl(\beta, r)$ is indeed closed, since in $\Compl(\beta)$ it is defined by the simultaneous vanishing of all $r_i\times r_i$ minors of the maps $A_i$.  We will then show that if $r'$ is obtained from $r$ by decreasing some $r_j$ by one, then $\Compl^\circ(\beta, r')$ is in the closure of $\Compl^\circ(\beta, r)$.  Thus, 
\[
	\Compl^\circ(\beta, r) \subset \bigcup\limits_{r'\preceq r} \Compl^\circ =\Compl(\beta, r) \subset \overline{\Compl^\circ(\beta, r)}.
	\]
It suffices to show that there is a continuous map $I^{(j)}_r: \C \rightarrow \Compl(\beta)$ such that $I^{(j)}_r(\nu) \in \Compl^\circ(\beta, r)$ for $\nu\neq 0$ and $I^{(j)}_r(0)\in \Compl^\circ(\beta, r-e_j)$.  Let $E_{j,j+1}(\nu)$ be the representation with 
\[
E_{j,j+1}(\nu)_i = \begin{cases} \C & \textrm{ if } i\in \{j, j+1\}\\ 0 & \textrm{ otherwise}\end{cases},\] and $E_{j, j+1}(\nu)(a_j): 1 \mapsto \nu$.  Define $$I^{(i)}_r(\nu) = \bigoplus\limits_{i\in [n]\setminus j} E_{i,i+1}^{r_i} \oplus E_{j, j+1}^{r_j-1} \oplus E_{j, j+1}(\nu) \oplus \bigoplus\limits_{i\in [n+1]}E_{i,i}^{s_i}.$$  Then $I_r^{(i)}(\nu) \in \Compl^\circ(\beta, r)$ for $\nu\in k^*$, and $I_r^{(i)}(0) \in \Compl^\circ(\beta, r-e_j)$.  This concludes the proof.
\end{proof}

\begin{proposition} \label{irr-cpts}
If $r$ is a maximal rank sequence, then $\Compl^\circ(\beta, r)$ is an open orbit (therefore irreducible), and $\Compl(\beta, r)$ is its closure.  Furthermore, $\Compl(\beta, r)$ is an irreducible component of $\Compl(\beta)$ and the set of all $\Compl(\beta, r)$ for $r$ a maximal rank sequence is a complete list of the irreducible components of $\Compl(\beta)$.  
\end{proposition}

\begin{proof}
For $r$ a fixed rank sequence, define the sets $R^{\geq}(\beta, r), R^{\leq}(\beta, r)\subset \bigoplus\limits_{i\in[n]} \Hom(V_i, V_{i+1})$ of elements $A=(A_j)_{j\in [n]}$ such that $\rk A_j \leq r_j$ (resp. $\rk A_j \geq r_j$) for $j\in [n]$.  Furthermore, we let $$R^{>}(\beta, r) = \bigcup\limits_{s\succ r} R^{\geq}(\beta, s) \subset R^{\geq}(\beta, r).$$  We will determine when such sets are closed or open by describing their defining equations.  For a pair of sets $(I_i, J_i) \subset [\beta_i]\times [\beta_{i+1}]$ with $\lvert I_i \rvert = \lvert J_i \rvert$, let $\Delta^{(i)}_{(I_i,J_i)}$ be the function on $\Compl(\beta)$ such that $\Delta^{(i)}_{(I_i, J_i)}\left( (A_j)_{j\in [n]}\right)$ is the minor of the matrix $A_i$ with columns given in order by $I_i$ and rows given by $J_i$.  For a rank sequence $r$, let $S^{(i)}_{r_i}$ be the set of all pairs of sets $(I_i, J_i)\subset [\beta_i]\times [\beta_{i+1}]$ with $\lvert I_i \rvert =\lvert J_i \rvert = r_i$.  Then we have the following descriptions of the sets above:
\begin{align*}
R^{\leq}(\beta, r) &= \bigcap\limits_{i\in [n]} \bigcap\limits_{(I_i, J_i)\in S^{(i)}_{r_i+1}} V(\Delta^{(i)}_{I_i, J_i})\\
R^{\geq}(\beta, r) &= \bigcap\limits_{i\in[n]} \bigcup\limits_{(I_i, J_i)\in S^{(i)}_{r_i}} NV(\Delta^{(i)}_{I_i, J_i})
\end{align*}
where $V(T)$ is the set of common zeros for the ideal $T$, and $NV(T)$ is the set of elements which do not vanish on all of $T$.  
 
From this description, $R^{\leq}(\beta, r)$ is closed, and $R^{\geq}(\beta, r)$ is open.  Note that $\Compl(\beta, r) = \Compl(\beta) \cap R^{\leq}(\beta, r)$ is closed in $\Compl(\beta)$.  Furthermore, 
\begin{align*}
\Compl^\circ(\beta, r) &= \Compl(\beta) \cap \left( R^{\geq}(\beta, r) \setminus R^{>}(\beta, r)\right)\\
				&= (\Compl(\beta) \cap R^{\geq}(\beta, r)) \setminus (\Compl(\beta) \cap R^{>}(\beta, r)).
				\end{align*}
However, if $r$ is a maximal rank sequence, then $\Compl(\beta) \cap R^{>}(\beta, r)$ is empty, for there are no complexes with larger rank sequence.  Thus, $\Compl^\circ(\beta, r) = \Compl(\beta) \cap R^{\geq}(\beta, r)$, which is open in $\Compl(\beta)$.  We have already seen that $\Compl^\circ(\beta, r)$ is a $\GL(\beta)$ orbit from proposition \ref{rank-sequence-proposition} and so it is irreducible.  Furthermore, $\Compl(\beta, r)$ is its closure, and clearly $\Compl(\beta)=\bigcup\limits_{r \textrm{ maximal}} \Compl(\beta, r)$ and $\Compl(\beta, r)\not\subset \Compl(\beta, r')$ whenever $r\neq r'$ are maximal rank sequences.  Therefore, this is a complete list of the irreducible components.  
\end{proof}

We have now shown that $\Compl(\beta, r)$ is a $\GL(\beta)$-variety.  This implies that $k[\Compl(\beta, r)]$ is a $\GL(\beta)$-module.  In the following section we will explore this structure.

\section{The Coordinate Ring $\C[\Compl_n(\beta, r)]$}\label{sec:TheCoordinateRing}
In this section, we describe the explicit basis of $\C[\Compl(\beta, r)]$ via minors prescribed by Young tableaux.  This actually illustrates a filtration on the coordinate ring whose associated graded is given by Schur modules.  The early portion of this section is a recollection of Young diagrams.  In the last part of this section, we describe a filtration on $\C[\Compl(\beta, r)]$ whose associated graded ring.  For the remainder of this section, we fix $n$, a dimension vector $\beta$, and a maximal rank sequence $r$ for $\beta$.  We will first set up notation for Young diagrams.

\begin{definition}
A \emph{Young diagram} $\lambda$ is a sequence of non-increasing positive integers $\lambda_1\geq \dotsc \geq \lambda_m$,  $m$ is called the number of parts of $\lambda$.
\end{definition}
We will draw Young diagrams as a table of rows of left-justified boxes such that the $i$-th row has $\lambda_i$ boxes.  For example, the diagram $(4,3,3,2)$ would be depicted by the figure
\begin{align*}
\ydiagram{4,3,3,2}
\end{align*}
For a Young diagram $\lambda$, we denote by $\lambda'$ the transpose diagram, where $\lambda'_i = \{i \mid \lambda_j\geq i\}$.  
\begin{definition}
Let $\lambda=\lambda_1\geq \dotsc \geq \lambda_m$, and $p$ be a positive integer with $p\geq m$.  Denote by $[p-\lambda]$ the diagram with $p$ parts and $[p-\lambda]_j = \lambda_1-\lambda_{p-j+1}$ (in this expression, if $\lambda_{p-j+1}$ is not defined, then it is considered to be $0$).
\end{definition}

\begin{definition}
A filling of $\lambda$ is an assignment of non-negative integers $t=\{t_{k,l}\}_{\substack{k=1,\dotsc, m\\ l=1,\dotsc, \lambda_k}}$ to the boxes of $\lambda$.  Such a filling is called \emph{standard} if $t_{k,l}<t_{k+1,l}$ whenever both are defined, and $t_{k,l}\leq t_{k, l+1}$ (i.e., the filling is column strictly increasing and row weakly increasing).  To a filling $t$ of $\lambda$, we associated a sequence of sets $I(t)=(I(t)_1,\dotsc, I(t)_{\lambda_1})$ where $I_l = \{t_{1,l}, t_{2,l}, \dotsc, t_{\lambda'_l, l}\}$.
\end{definition}
\begin{example}\label{ex:youngdiagram}
Again, consider the diagram $\lambda=(4,3,3,2)$, so $\lambda'=(4,4,3,1)$.  If $p=6$, then $[p-\lambda]=(4,4,2,1,1,0)$.  This is most easily calculated by placing $\lambda$ in the top-left corner of a $p\times \lambda_1$ rectangle
\begin{align*}
\ytableaushort{\none, \none, \none}*{4,4,4,4,4,4}*[\bullet]{4,3,3,2}
\end{align*}
and $[p-\lambda]$ is the diagram left over in the bottom-right corner. 

The following is a standard filling $t$ of $\lambda$.
\begin{align*}
\begin{ytableau}
1 & 2 & 2 & 4\\
2 & 4 & 5 \\
3 & 6 & 7\\
7 & 7
\end{ytableau}
\end{align*}
and the associated sequence of sets is $I(t)=(\{1,2,3,7\}, \{2,4,6\}, \{2,5,7\}, \{4\})$.  
\end{example}

\begin{definition}
Let $V$ be a vector space, and $\lambda$ a Young diagram with at most $\dim V$ parts.  We will denote by $\bigwedge^\lambda V$ the product of exterior powers of $V$ prescribed by the columns of $\lambda$.  Namely 
$$\bigwedge^\lambda V = \bigwedge^{\lambda'_1} V \otimes \dotsc \otimes \bigwedge^{\lambda'_{\lambda_1}} V.$$
\end{definition}
\begin{remark} 
Notice that the basis of $\bigwedge^\lambda V$ is in one-to-one correspondence with column-increasing fillings of $\lambda$ consisting of integers $1, \dotsc, \dim V$.  For a set $I = \{i_1,\dotsc, i_k\}$, let $e_I =e_{i_1} \wedge \dotsc \wedge e_{i_k}$.  If $t$ is a column-increasing filling of $\lambda$ with integers from the set $\{1,\dotsc, \dim V\}$, then we have the associated sequence of sets $(I(t)_1,\dotsc, I(t)_{\lambda_1})$ and the associated basis element in $\bigwedge^\lambda V$ is $e_{I(t)_1} \otimes \dotsc \otimes e_{I(t)_{\lambda_1}}$.  
\end{remark}
\begin{example}\label{ex:youngdiagramsmall}
Taking $\lambda=(3,2,2)$, the filling
\begin{align*}
t=\begin{ytableau}
1 & 2 & 4 \\ 
4 & 3 \\
5 & 4 
\end{ytableau}
\end{align*}
corresponds to the element $(e_1\wedge e_4 \wedge e_5) \otimes (e_2\wedge e_3 \wedge e_4) \otimes e_4$ in the space 
\begin{align*}
\bigwedge^\lambda V = \bigwedge^3 V \otimes \bigwedge^3 V \otimes \bigwedge^1 V
\end{align*}
when $\dim V \geq 5$.  
\end{example}

We will speak of fillings of $\lambda$ and basis elements of $\bigwedge^\lambda V$ interchangeably. 
\begin{definition}
Let $V$ be a $\C$-vector space and $\lambda$ a Young diagram with at most $\dim V$ parts.  Define the map $$\op_\lambda: \bigwedge^{[\dim V-\lambda]} V \rightarrow \bigwedge^{\lambda} V$$ as follows: if $t$ is a column-increasing filling of $[\dim V - \lambda]$, and $I(t) = (I(t)_1,\dotsc, I(t)_{\lambda_1})$ is the associated sequence of sets, then take $t'$ the filling of $\lambda$ with associated sequence of sets $I(t') = (I(t')_1, \dotsc, I(t')_{\lambda_1})$ such that $I(t')_j = \{1,\dotsc, \dim V\} \setminus I(t)_{\lambda_1-j+1}$.  Then $$\op_\lambda (t) := \left( \prod\limits_{j=1}^{\lambda_1} \sgn(I(t')_j, I(t)_{\lambda_1-j+1})\right) t'.$$  Here, $\sgn(I, J)$ is the sign of the permutation $(I, J)$ with both $I, J$ written in increasing order.  
\end{definition}

\begin{example}
Suppose that $\dim V = 5$, and $\lambda = (3,2,2)$ as in example \ref{ex:youngdiagramsmall}.  Then $[\dim V - \lambda] = (3,3,1,1,0)$.  Take $t$ the filling of $[\dim V-\lambda]$ given by 
\begin{align*}
t=\begin{ytableau}
1 & 1 & 2 \\
2 & 5 & 3 \\
3 \\ 
5 \end{ytableau}
\end{align*}
then 
\begin{align*}
t'= \begin{ytableau}
1 & 2 & 4 \\
 4 & 3 \\ 
 5 & 4 \end{ytableau}
\end{align*}
Since $\sgn(14523)=1$, $\sgn(23415)=-1$ and $\sgn(12354)=-1$, we have that $$\op_\lambda (t) = t'.$$
\end{example}

\begin{definition}
Suppose that $V_i, V_{i+1}$ are $\C$-vector spaces, and $\lambda$ is a Young diagram with at most $\min(\dim V_i, \dim V_{i+1})$ parts.  Define a map $\delta^{(i)}_\lambda: \bigwedge^{\lambda} V_i \otimes \bigwedge^{\lambda} V_{i+1} \rightarrow \C[\Compl(\beta, r)]$ as follows:  suppose that $t_i$ is a filling of $\lambda$ from the integers $\{1,\dotsc, \dim V_i\}$ with associated sequence of sets $I(t_i)$, and $t_{i+1}$ is a filling of $\lambda$ from the integers $\{1,\dotsc, \dim V_{i+1}\}$ with associated sequence of sets $I(t_{i+1})$.  Then $$\delta_\lambda:t_i \otimes t_{i+1} \mapsto \prod\limits_{j=1}^{\lambda_1} \Delta^{(i)}_{I(t_i)_j, I(t_{i+1})_j}.$$  Recall that $\Delta_{I, J}^{(i)}$ is the minor of the matrix $A_i$ with columns given by $I$ and rows given by $J$.  

If $\lambda=(\lambda(1),\dotsc, \lambda(n))$ is a sequence of Young diagrams such that $\lambda(i)$ has at most $\min(\dim V_i, \dim V_{i+1})$ parts, then take 
\begin{align*}
\delta_\lambda: \bigotimes\limits_{i=1}^n \left( \bigwedge^{\lambda(i)} V_i \otimes \bigwedge^{[\beta_{i+1}-\lambda(i)]} V_{i+1}\right) \rightarrow \C[\Compl(\beta, r)]
\end{align*}
To be the composition of the map
\begin{align*}
\bigotimes\limits_{i=1}^n (\operatorname{id} \otimes \op_{\lambda(i)}): \bigotimes\limits_{i=1}^n \left( \bigwedge^{\lambda(i)} V_i \otimes \bigwedge^{[\beta_{i+1}-\lambda(i)]} V_{i+1}\right)  \rightarrow \bigotimes\limits_{i=1}^n \left( \bigwedge^{\lambda(i)} V_i \otimes \bigwedge^{\lambda(i)} V_{i+1}\right) 
\end{align*}
And the map 
\begin{align*}
\bigotimes\limits_{i=1}^n \delta_\lambda^{(i)}: \bigotimes\limits_{i=1}^n \left( \bigwedge^{\lambda(i)} V_i \otimes \bigwedge^{\lambda(i)} V_{i+1}\right)  \rightarrow \C[\Compl(\beta, r)].
\end{align*}
\end{definition}

\begin{example}
Suppose that $n=2$, $\beta=(2,5,3)$, and $r=(2,3)$.  Let $\lambda = ( (2,1), (2,2,1) )$.  Then $\delta_\lambda$ is a map
\begin{align*}
\bigwedge^{(2,1)} V_1 \otimes \bigwedge^{(2,2,2,1)} V_2 \otimes \bigwedge^{(2,2,1)} V_2 \otimes \bigwedge^{(1)} V_3 \rightarrow \C[\Compl(\beta, r)].
\end{align*}
As for an explicit calculation,
\begin{align*}
(\bigotimes\limits_{i=1}^n \delta_\lambda^{(i)})\circ (\bigotimes\limits_{i=1}^n (\operatorname{id} \otimes \op_{\lambda(i)}))\left(\begin{ytableau} 1 & 2 \\ 2 \end{ytableau}\otimes \begin{ytableau} 2 & 1\\ 3 & 2 \\ 4 & 4 \\ 5 \end{ytableau}\right) \otimes& \left(\begin{ytableau} 1 & 2 \\ 2 & 3 \\ 4 \end{ytableau} \otimes \begin{ytableau} 3 \end{ytableau}\right)\\
=\bigotimes\limits_{i=1}^n \delta_\lambda^{(i)}\left(\begin{ytableau} 1 & 2 \\ 2 \end{ytableau}\otimes \begin{ytableau} 3 & 1 \\ 5 \end{ytableau}\right) \otimes& \left( \begin{ytableau} 1 & 2 \\ 2 & 3 \\ 4 \end{ytableau} \otimes \begin{ytableau} 1 & 1 \\ 2 & 2 \\ 3  \end{ytableau}\right)\\
=\Delta_{12,35}^{(1)}\Delta_{2,1}^{(1)} \Delta_{124,123}^{(2)} \Delta_{23,12}^{(2)}&
\end{align*}
\end{example}

\begin{remark}
If $\lambda(i)$ has more than $r(i)$ parts for some $i$, then $\operatorname{image} \delta_\lambda=0$ on $\Compl(\beta, r)$ since one factor is the an $r(i)+l\times r(i)+l$ minor of $A_i$, and $\rk A_i \leq r(i)$ by definition of $\Compl(\beta, r)$.  
\end{remark}

\begin{definition}\label{def:lambdabracket}
Let $\Lambda_n(\beta, r)$ be the set of sequence of partitions $(\lambda(1),\dotsc, \lambda(n))$ such that $[\beta_{i+1}- \lambda(i)]_{\lambda(i)_1} \leq \lambda(i+1)'_1$.  I.e., the first column of $[\beta_{i+1}-\lambda(i)]$ is shorter than the last column of $\lambda(i+1)$.  If $\lambda \in \Lambda_n(\beta, r)$, denote by $[\lambda(i+1): \lambda(i)]$ the Young diagram with $$[\lambda(i+1):\lambda(i)]_j = [\beta_{i+1}-\lambda(i)]_j +\lambda(i+1)_j.$$  Diagrammatically, this is simply juxtaposing the diagrams $\lambda(i)$ and $[\beta_{i+1}-\lambda(i)]$, which is still a Young diagram by definition of $\Lambda_n(\beta, r)$.  We will also write $\lambda(1)=[\lambda(1):\lambda(0)]$ and $[\beta_{n+1}-\lambda(n)]=[\lambda(n+1):\lambda(n)]$ for the degenerate cases.

A filling of the diagrams $[\lambda(1):\lambda(0)], [\lambda(2):\lambda(1)], \dotsc, [\lambda(n):\lambda(n-1)], [\lambda(n+1):\lambda(n)]]$ is the same as a filling of all diagrams $\lambda(i)$ and $[\beta_{i+1}-\lambda(i)]$ for $i=1,\dotsc, n$ and is called a \emph{multitableau}.  
\end{definition}
\begin{example}
Again, take $n=2, \beta=(2,5,3), r=(2,3)$.  If $\lambda(1)=(2,1)$ and $\lambda(2)=(2,1)$, then $[\beta_2-\lambda(1)]=(2,2,2,1,0)$ and $[\beta_3-\lambda(2)]=(2,1,0)$.  Then 
\begin{align*}
[\lambda(2):\lambda(1)] = \ydiagram{4,3,2,1}
\end{align*}
The following is a multitableau of shape $\lambda$:
\begin{align*}
\begin{ytableau} 1 & 2 \\ 1 \end{ytableau}\qquad \begin{ytableau} 1 & 2 & *(lightgray)3 & *(lightgray)3 \\ 2 & 3 & *(lightgray)4\\ 3 & 5 \\ 5 \end{ytableau}\qquad \begin{ytableau} 2 & 3 \\ 3 \end{ytableau}
\end{align*}

The corresponding element in $\bigwedge^{\lambda(1)} V_1 \otimes \bigwedge^{[\beta_2-\lambda(1)]} V_2 \otimes \bigwedge^{\lambda(2)} V_2 \otimes \bigwedge^{[\beta_3-\lambda(2)]} V_3$ is 
\begin{align*}
\begin{ytableau} 1 & 2 \\ 1 \end{ytableau}\otimes \begin{ytableau} 1 & 2 \\ 2 & 3 \\ 3 & 5 \\ 5 \end{ytableau} \otimes \begin{ytableau} 3 & 3 \\ 4 \end{ytableau} \otimes \begin{ytableau} 2 & 3\\ 3 \end{ytableau} 
\end{align*}

\end{example} 

\begin{definition}
For two partitions $\lambda, \mu$, we define $\lambda \preceq \mu$ if $(\lambda'_1,\dotsc, \lambda'_{\lambda_1})\geq (\mu'_1,\dotsc, \mu'_{\mu_1})$.  Extend this to a partial order on $\Lambda_n(\beta, r)$ with $\lambda \preceq \mu$ if 
\begin{align*}
([\lambda(1):\lambda(0)], [\lambda(2):\lambda(1)], \dotsc, [\lambda(n+1):\lambda(n)], [\lambda(n+1):\lambda(n)])\\ \phantom{e} \preceq ([\mu(1):\mu(0)], [\mu(2):\mu(1)], \dotsc, [\mu(n+1): \mu(n)], [\mu(n+1):\mu(n)])
\end{align*}
 in the lexicographical order.
\end{definition}

\begin{definition}
Suppose that $\lambda$ and $\mu$ are partitions with $r_1, r_2$ parts, respectively, and $V$ is a vector space of dimension $n$.  Then we write $S_{(\lambda, -\mu)}V$ to denote the Schur module $S_{(\lambda_1,\dotsc, \lambda_{r_1}, 0,\dotsc, 0, -\mu_{r_2}, -\mu_{r_2-1}, \dotsc, -\mu_1)} V$, where we include $n-(r_1+r_2)$ zeros in the indexing vector.  Furthermore, we will write $-\mu$ for the vector $(-\mu_{r_2}, \dotsc, -\mu_2, -\mu_1)$.  
\end{definition}

\begin{proposition}[\cite{DCS}]
Denote by $\mathcal{F}_\lambda=\sum\limits_{\substack{\mu \in \Lambda_n(\beta, r)\\ \mu \preceq \lambda}} \operatorname{image } \delta_\mu$, and $\mathcal{F}_{\prec \lambda} = \sum\limits_{\substack{\mu\in \Lambda_n(\beta, r)\\ \mu \prec \lambda}} \operatorname{image } \delta_\mu$.  Then $\mathcal{F}_\lambda / \mathcal{F}_{\prec \lambda}$ has a basis given by \emph{standard} fillings of the diagrams $[\lambda(i+1):\lambda(i)]$ for $i=0,\dotsc, n$.  A collection of fillings of this sequence of diagrams is called a \emph{multitableau of shape $\lambda$}.  Furthermore, $$\mathcal{F}_\lambda/\mathcal{F}_{\prec \lambda} \cong \bigotimes\limits_{i=1}^n S_{(\lambda(i), 0, \dotsc, 0, -\lambda(i-1))} V_i.$$
\end{proposition}

The above proposition is proven by showing that if $t_\lambda$ is a multitableau of shape $\lambda$, then $\delta_\lambda(t_\lambda)$ can be written, modulo terms in $\mathcal{F}_{\prec \lambda}$, as a linear combination of standard multitableau of shape $\lambda$.  

\begin{definition} The \emph{content} of a multitableau $t$ of shape $\lambda$ is the sequence of vectors $(\kappa^1,\dotsc, \kappa^{n+1})$ where 
\begin{align*}
\kappa^i_j &= \#\{\textrm{ boxes in } [\lambda(i):\lambda(i-1)] \textrm{ that are filled with the integer } j\}\\
\end{align*}
\end{definition}

\begin{corollary}[\cite{DCS}]
Suppose that $t$ is a non-standard multitableau of shape $\lambda$.  Then
\begin{align*}
\delta_\lambda(t) &= s(t) + y(t)
\end{align*}
where $s(t)$ is a linear combination of standard multitableaux of the same content as $t$, and $y(t)\in \mathcal{F}_{\prec \lambda}$.
\end{corollary}

\begin{proposition}[\cite{DCS}]
$\C[\Compl(\beta, r)] =\bigcup\limits_{\lambda \in \Lambda_n(\beta, r)} \mathcal{F}_\lambda$.  
\end{proposition}

\begin{definition}\label{def:partitionsum}
Suppose that $\lambda, \mu \in \Lambda_n(\beta, r)$.  Define by $\lambda+\mu$ the sequence of diagrams with $(\lambda+\mu)(i)_j = \lambda(i)_j + \mu(i)_j.$
\end{definition}

\begin{proposition}\label{prop:semigroupcomplex}
Suppose that $t_\lambda$ and $t_\mu$ are multitableaux of shapes $\lambda$ and $\mu$, respectively.  Then $$\delta_\lambda(t_\lambda) \cdot \delta_\mu(t_\mu) \in \mathcal{F}_{\lambda+\mu}.$$
\end{proposition}

\begin{proof}
It suffices to show this when $\mu$ consists of a single column, i.e., $\mu(i)=(\overbrace{1,1,\dotsc, 1}^j, 0, 0,\dotsc)$ for some $i$ and $\mu(j)=0$ otherwise.  Thus, $\delta_\lambda(t_\mu) = \Delta_{I, J}^{(i)}$ for some sets $I\subset \{1,\dotsc, \beta_i\}$, $J\subset \{1, \dotsc, \beta_{i+1}\}$.  Therefore, $\delta_\lambda(t_\lambda)\cdot \delta_\mu(t_\mu) = \delta_\lambda(t_\lambda) \cdot \Delta_{I, J}^{(i)}$.  Now  notice that $\lambda+\mu$ is the sequence diagrams which is the same as $\lambda$ except for $(\lambda+\mu)(i)$ which has an extra column of height $j$.  Take multitableau of shape $(\lambda+\mu)$ so that all entries not corresponding to the extra column are the same as in the filling $t_\lambda$, and all entries in the columns corresponding to the extra column are taken from $t_\mu$.  Denoting by $t_{\lambda+\mu}$ this filling, we have that $\delta_{\lambda+\mu}(t_{\lambda+\mu}) = \delta_\lambda (t_\lambda) \cdot \delta_\mu (t_\mu)$.  In short, each column of the sequence $\mu$ can be absorbed into $\lambda$ until the result is the sequence $\lambda + \mu$.   
\end{proof}

\begin{corollary}[\cite{DCS}]
The set $\{\mathcal{F}_\lambda\}_{\lambda \in \Lambda}$ gives a filtration of $\C[\Compl(\beta, r)]$ and the associated graded algebra is 
\begin{align*}
\gr_\Lambda \left(\C[\Compl(\beta, r)]\right) & = \bigoplus\limits_{\lambda \in \Lambda} \bigotimes\limits_{i=1}^n S_{(\lambda(i), 0, 0, \dotsc, 0, -\lambda(i-1))} V_i.
\end{align*}
(For the definition of the Schur modules $S_\lambda V$, we refer to \cite{W}[Chapter 2.1].)
\end{corollary}

We are now prepared to consider the variety $\Rep_{\C Q/I}(\beta)$ when $\C Q/I$ is a gentle string algebra.


\section{The Coordinate Rings of $\Rep_{\C Q/I}(\beta)$}\label{sec:coordinateringofrep}
\newcommand{\X}{\mathfrak{X}}
Let $\C Q/I$ be a gentle string algebra with $c:Q\rightarrow S$ the coloring of $Q$ such that $I=I_c$.  Let $\beta$ a dimension vector for $Q$.  We will determine the irreducible components of $\Rep_{\C Q/I}(\beta)$, and determine the rings $\SI_{Q, c}(\beta)_i$ via generators and relations.  This analysis will yield we find upper bounds for the degrees of the generators and relations viewed as polynomials in $\C[\Rep_{\C Q/I}(\beta)_i]$.  For brevity we will write $\Rep_{Q, c}(\beta)$ for the space $\Rep_{\C Q/I_c}(\beta)$, and $\SI_{Q, c}(\beta)$ for the subring of semi-invariants.

Define by $\X \subset Q_0 \times S$ the space of pairs with $(x, s) \in \X$ if and only if $s$ is a color passing through $x$, i.e., there is an arrow $a$ with $h(a)=x$ or $t(a)=x$ and $c(a)=s$.  For such a pair, we denote by $i(x, s)$ (resp. $o(x, s)$) the arrow of color $s$ whose head (resp. tail) is $x$.  Formally, we write $\emptyset$ if this arrow doesn't exist.  A vertex will be called \emph{lonely} if there is only one element $(x,s)\in \X$, and \emph{coupled} if there are two (notice that there can be at most $2$).

\begin{definition}
A map $r: Q_1 \rightarrow \N$ is called a \emph{rank sequence} if its restriction to each colored path is a rank sequence.  I.e., $r(i(x, s))+r(o(x,s)) \leq \beta_x$  for all $(x, s)\in \X$ (here we adopt the convention that $r(\emptyset) = 0$). The rank sequence $r$ is called \emph{maximal} if it is so under the coordinate-wise partial order.
\end{definition}

\begin{proposition}
Let $\Rep_{Q,c}(\beta, r) = \{M\in \Rep_{Q, c}(\beta) \mid \rk(M(a))\leq r(a) \ \forall a\in Q_1\}.$  This variety is irreducible if and only if $r$ is maximal under the coordinate-wise partial ordering.  Furthermore, for such rank sequence $r$, $\Rep_{Q, c}(\beta, r)$ is normal and Cohen-Macaulay with rational singularities..
\end{proposition}

\begin{proof}
This is clear once we show that $\Rep_{Q, c}(\beta)$ is the product of varieties of complexes (the normality, Cohen-Macaulay and rational singularity properties are corollaries of \cite{DCS}).  More specifically, for each color $s\in S$, let $X_s$ be the set of vertices such that $(x,s)\in \X$, $A_s$ the set of arrows $a$ with $c(a)=s$, $\beta^s$ be the restriction of the dimension vector $\beta$ to the vertices $X_s$, and $r^s$ the restriction of the rank sequence $r$ to the arrows in $A_s$.  Then as varieties we have $$\prod\limits_{s\in S} \Compl_{\lvert X_s \rvert} (\beta^s, r^s)\cong \Rep_{Q, c}(\beta, r).$$  
Therefore $$\C [\Rep_{Q, c}(\beta, r)]\cong \bigotimes\limits_{s\in S} \C [\Compl_{\lvert X_s\rvert} (\beta^s, r^s)].$$
\end{proof}

We now describe a filtration on $\C [\Rep_{Q, c}(\beta, r)]$ by collections of Young Diagrams.

\begin{definition}
Let $\Lambda(Q, c, \beta, r)$ be the set of functions $\lambda: Q_1 \rightarrow \mathcal{P}$ (where $\mathcal{P}$ is the set of Young diagrams) such that $\lambda(a)$ has at most $r(a)$ non-zero parts (this is the generalization of $\Lambda_n(\beta, r)$ from section \ref{sec:TheCoordinateRing}).  We will simply write $\Lambda$ if the parameters are understood.  If $p_s=p_{m_s}\dotsc p_1$ is the full path of color $s$, then we write $\lambda^s$ for the sequence $\lambda(p_1), \dotsc, \lambda(p_{m_s})$.
\end{definition}

As a result of the previous section, for each $\lambda \in \Lambda$, we have the map 

\begin{align}\label{map:deltahat}
\hat{\delta}_\lambda: \bigotimes\limits_{s\in S} \left(\bigotimes\limits_{i=1}^{m_s} \bigwedge^{\lambda(p_i)} V_{t_{p_i}} \otimes \bigwedge^{[\beta_{h_{p_i}}-\lambda(p_i)]} V_{h_{p_i}}\right) \rightarrow \C [\Rep_{Q, c}(\beta,r)] \end{align} which is given by the product of the maps $\hat{\delta}_{\lambda^s}$.  It will be convenient to denote the domain of this map by $\bigwedge^\lambda V$.  Furthermore, we may extend the partial order of \ref{sec:TheCoordinateRing} to $\Lambda$ as follows: $$\lambda \preceq \mu$$ if and only if $$\lambda^s \preceq \mu^s$$ for each color $s\in S$, where $\preceq$ on colored sequences is given by that for the variety of complexes associated to that color.  Finally, we denote by $\mathcal{F}_\lambda = \sum\limits_{\mu\preceq \lambda} \operatorname{image}(\hat{\delta}_\mu),$ and $\mathcal{F}_{\prec \lambda} = \sum\limits_{\mu \prec \lambda} \operatorname{image}(\hat{\delta}_\mu)$.

\begin{proposition}
For $r$ a maximal rank sequence for $\beta,$ the collection $\{\mathcal{F}_\lambda\mid \lambda\in \Lambda(Q, c, \beta, r)\}$ is a filtration of $\Rep_{Q, c}(\beta, r)$ relative to the partial order just described.  Furthermore, $$\mathcal{F}_{\lambda} /\sum\limits_{\mu\prec \lambda} \mathcal{F}_\mu\cong \bigotimes\limits_{(x,s)\in \X} S_{\lambda(x,s)} V_x$$ where $\lambda(x,s) = (\lambda(o(x,s)),-\lambda(i(x,s)))$.  
\end{proposition}

\begin{proof} As the product of spaces that are filtered by sequences of partitions, $\C[\Rep_{Q, c}(\beta, r)$ is filtered by collections of sequences of partitions.  
\end{proof}

\begin{corollary}
For $Q, c, \beta, r$ as above, \begin{align}\label{ASSOCIATEDGRADED:BYX}\gr(\C[\Rep_{Q, c}(\beta, r)]) &\cong \bigoplus\limits_{\lambda\in \Lambda} \bigotimes\limits_{(x, s)\in \X} S_{\lambda(x, s)} V_x\\
& \cong \bigoplus\limits_{\lambda \in \Lambda} \bigotimes\limits_{x\in Q_0} S_{\lambda(x,s_1(x))} V_x \otimes S_{\lambda(x,s_2)}V_x.\end{align}
\end{corollary}

As a result of the above remarks, we can give a basis for $\C [\Rep_{Q, c}(\beta, r)]$ via standard multitableaux by generalizing the procedure described by DeConcini and Strickland in the case of the varieties of complexes.  The following definitions will serve as the notation necessary for such a generalization.

\begin{definition}\label{def:bracketnotation}
Let $\lambda\in \Lambda(Q, c, \beta, r)$.  For each element $(x,s)\in \X$, denote by $[\lambda_{x,s}]$ the partition with $[\lambda_{x,s}]_j = \lambda(o(x,s))_j + (\beta_x-\lambda(i(x,s)))_j$.  This can be viewed as adjoining the partitions $(\beta_x-\lambda(i(x,s))$  and $\lambda(o(x,s))$ left-to-right.  
\end{definition}
This is the natural generalization of the symbol $[\lambda(i):\lambda(i-1)]$ in section \ref{sec:TheCoordinateRing}, so we expect to build a basis from fillings of these diagrams.  

\begin{definition}
Let $\lambda \in \Lambda(Q, c,\beta, r)$.  A \emph{multitableau of shape $\lambda$} is a column-strictly-increasing filling of each of the diagrams $[\lambda_{x,s}]$ for $(x, s)\in \X$.  A multitableau is called \emph{standard} if each filling of each diagram is a standard filling.  The \emph{content} $\kappa$ of a filling of $\lambda$ is the collection of vectors $\kappa_{x,s}\in \N^{\beta_x}$ with $(\kappa_{x,s})_j=\#\{\textrm{ occurrences of } j \textrm{ in the filling of } [\lambda_{x,s}]\}$.
\end{definition}

Using the same conventions as in section \ref{sec:TheCoordinateRing}, we can see that $\bigwedge^\lambda V$ has basis given by multitableaux of shape $\lambda$.  In the subsequent section, we will determine explicit elements of $\bigwedge^\lambda V$ whose image under $\hat\delta_\lambda$ is a semi-invariant function.


\section{Semi-Invariant Functions in $\C [\Rep_{Q, c}(\beta, r)]$}\label{sec:SIonREP}
Fix a colored string algebra $(Q, c)$, a dimension vector $\beta$, and a maximal rank sequence $r$ for $\beta$.  We denote by $M_\lambda$ the term $\mathcal{F}_\lambda/\mathcal{F}_{\prec \lambda}$ for $\lambda \in \Lambda$.  With this notation, we may write $\gr(\C [\Rep_{Q, c}(\beta, r)]) \cong \bigoplus\limits_{\lambda \in \Lambda} M_\lambda.$  In the forthcoming, we will show that $\SI_{Q, c}(\beta, r)$ is isomorphic to a semigroup ring.  We do so by defining a basis $\{m_\lambda\}$ for $\SI_{Q, c}(\beta, r)$ and then exhibiting the multiplication on said basis.
\begin{definition} Let $\Lambda_{SI}(Q, c, \beta, r)$ be the set of elements $\lambda$ in $\Lambda(Q, c, \beta, r)$ such that $M_\lambda$ contains a semi-invariant.  As usual, we write $\Lambda_{SI}$ if the parameters are understood.  \end{definition}
\begin{proposition}\label{prop:cauchy}
Let $\lambda \in \Lambda$.  Then $\lambda \in \Lambda_{SI}$ if and only if there is a vector $\sigma(\lambda) \in \mathbb{Z}^{Q_0}$ such that for each $x\in Q_0$, we have 
\begin{align} 
\lambda(x,s_1)_i + \lambda(x,s_2)_{\beta_x+1-i} &= \sigma(\lambda)_x \qquad i=1,\dotsc, \beta_x,
\end{align}
(here if $x$ is a lonely vertex, then the second summand is suppressed, i.e., $\lambda(x,s)_i = \sigma(\lambda)_x$ for $i=1,\dotsc, \beta_x$.  Furthermore, if $\lambda \in \Lambda_{SI}$, then the space of semi-invariants in $M_\lambda$ is one-dimensional.
\end{proposition}
\begin{proof}
The decomposition of the tensor product of two Schur modules is given by the Littlewood-Richardson rule (cf. [\cite{W} proposition 2.3.1]).  Applying this to equation \ref{ASSOCIATEDGRADED:BYX}, we see that there is an $\SL(\beta)$-invariant (meaning that the Schur module appearing as a factor at $x$ is a height-$\beta_x$ rectangle for each $x$) if and only if the system of equations in the proposition hold.  
\end{proof}
\begin{corollary}
$\Lambda_{SI}$ is a semigroup under the $+$ operation as defined on partitions in \ref{def:partitionsum}.
\end{corollary}
\begin{proof}
Indeed, if $\sigma(\lambda)$ and $\sigma(\mu)$ are the vectors in $\Z^{Q_0}$ satisfying proposition \ref{prop:cauchy} for the sequences $\lambda, \mu \in \Lambda$, then $\sigma(\lambda)+\sigma(\mu)$ is the vector satisfying the proposition for the sequence $\lambda+\mu$.  
\end{proof}

\begin{remark} Recall the definition of $[\lambda_{x,s}]$ in \ref{def:bracketnotation}.  We will collect some useful points:  
\begin{itemize}
\item[a.] this notation allows us to rewrite the domain of the map \ref{map:deltahat} in the form 
\begin{align*}
\bigotimes\limits_{(x,s)\in \X} \bigwedge^{[\lambda_{x,s}]} V_x
\end{align*}
\item[b.] Using this notation, we can restate proposition \ref{prop:cauchy}, namely that $\lambda \in \Lambda_{SI}$ if and only if there is a vector $\sigma(\lambda) \in \Z^{Q_0}$ such that
\begin{itemize}
\item[i.] For each lonely element $(x,s)\in \X$ (i.e., with no other color passing through $x$), $[\lambda_{x,s}]'_i = \beta_x$, $i=1,\dotsc, \sigma(\lambda)_x$
\item[ii.] For each coupled pair $(x,s_1), (x,s_2)\in \X$, $[\lambda_{x,s_1}]'_i +[\lambda_{x,s_2}]'_{\sigma_x-i+1} = \beta_x$ for $i=1,\dotsc, \sigma(\lambda)_x$.
\end{itemize}
\end{itemize}
\end{remark}
This restatement will be useful for defining a map whose image consists of semi-invariants.  
\begin{definition}
For $\lambda\in \Lambda_{SI}$, define the following maps:
\begin{itemize}
\item[i.] If $(x,s)\in \X$ is a lonely pair, then let $$\Delta_i^{\lambda, x}: \bigwedge^{\beta_x} V_x \rightarrow \bigwedge^{[\lambda_{x,s}]'_i} V_x$$ be the identity map for $i=1,\dotsc, \sigma(\lambda)_x$ (since, by the above remark, $\beta_x=[\lambda_{x,s}]'_i$);
\item[ii.] If there is a coupled pair $(x,s_1),\ (x,s_2)\in \X$, then take $$\Delta_i^{\lambda, x}: \bigwedge^{\beta_x} V_x \rightarrow \bigwedge^{[\lambda(x,s_1)]'_i} V_x \otimes \bigwedge^{[\lambda(x,s_2)]'_{\sigma(\lambda)_x-i+1}} V_x$$ to be the diagonalization map (since, by the above remark, the sum of the two powers is precisely $\beta_x$).  
\end{itemize}
We collect these maps into the map $\Delta^\lambda$ in the following way:
\begin{align}\label{map:Deltalambda}
\Delta^\lambda:= \bigotimes\limits_{x\in Q_0} \bigotimes\limits_{i=1}^{\sigma(\lambda)_x} \Delta_i^{\lambda, x}: \left( \bigotimes\limits_{x\in Q_0} \left( \bigwedge^{\beta_x} V_x\right)^{\sigma(\lambda)_x}\right) \rightarrow \bigotimes\limits_{(x,s)\in \X} \bigwedge^{[\lambda_{x,s}]} V_x.
\end{align}
\end{definition}

Notice that $\Delta^\lambda$ is a $\GL(\beta)$-equivariant map, since both identity and diagonalization are such.  Fixing a basis for each space $V_x$, and let $\underline{e}$ be the corresponding basis element of $\bigotimes\limits_{x\in Q_0} \left(\bigwedge^{\beta_x} V_x\right)^{\sigma(\lambda)_x}$ (note that this space is one-dimensional).  
\begin{definition}
Denote by $$m_\lambda = \hat{\delta}_\lambda \Delta^\lambda(\underline{e}).$$  This is unique up to scalar multiple.
\end{definition}

\begin{proposition}\label{prop:explicitsemiinvariants}
For $\lambda \in \Lambda_{SI}$, the function $m_\lambda$ is a semi-invariant of weight $\sigma(\lambda)$.  Furthermore, $\overline{m}_\lambda \neq 0 \in \mathcal{F}_\lambda/\mathcal{F}_{\prec \lambda}$.
\end{proposition}
The first statement is evident since both $\hat{\delta}_\lambda$ and $\Delta^\lambda$ are $\GL(\beta)$-equivariant homomorphisms, and the weight is clear from the action on the domain of the map.  We delay the proof of the second statement for a brief description of the straightening relations in $\C [\Rep_{Q, c}(\beta, r)]$ relative to fillings of Young diagrams, since the description of $m_\lambda$ is not given in terms of \emph{standard} multitableaux.  We will come back to this proof when we can show that there is a \emph{standard} multitableau of shape $\lambda$ whose coefficient is non-zero in $m_\lambda$.  The following is simply a generalization of the material in section \ref{sec:TheCoordinateRing}.  We record these statements as corollaries to DeConcini and Strickland.

\begin{corollary}[\cite{DCS}]
If $t_\lambda$ is a filling of $\lambda$, then 
$$\hat\delta_\lambda(t_\lambda) = s(t_\lambda)+y(t_\lambda)$$ where $y(t_\lambda) \in \mathcal{F}_{\prec \lambda}$ and $s(t_\lambda)$ is a linear combination of standard fillings of the same content as $t_\lambda)$.  
\end{corollary}

\begin{corollary}[\cite{DCS}]
If $t_\lambda$ and $t_\mu$ are fillings of shape $\lambda, \mu$, then $$\hat\delta_\lambda(t_\lambda) \cdot \hat\delta_\mu(t_\mu) \in \mathcal{F}_{\lambda+\mu}.$$
\end{corollary}
\begin{proof}[proof of proposition \ref{prop:explicitsemiinvariants}]
It remains to be shown that $\overline{m}_\lambda \neq 0$ in $\mathcal{F}_{\lambda}/\mathcal{F}_{\prec \lambda}$.  For a filling $t_\lambda$ of $\lambda$, let $I(t_\lambda)_{x,s, i}$ be the set of entries in the $i$-th column of $[\lambda_{x,s}]$.  Notice that $\Delta^\lambda(\underline{e})$ is the sum of all fillings $t_\lambda$ of $\lambda$ satisfying the property that $I(t_\lambda)_{x,s,i}\cup I(t_\lambda)_{x,s', \sigma(\lambda)_x-i+1} = \{1,\dotsc, \beta_x\}$, call this property $(\ast)$.  Pick one distinguished element from each coupled pair $(x,s),\ (x,s')\in \X$.  Consider the filling $t^\circ_\lambda$ of $\lambda$ with $I(t^\circ_\lambda)_{x,s,i} = \{1,2, \dotsc, [\lambda_{x,s}]'_i\}$ whenever $(x,s)$ is the distinguished element in the coupled pair and $I(t^\circ_\lambda)_{x,s', i} = \{\beta_x, \beta_x-1,\dotsc, [\lambda_{x,s'}]'_i\}$.  This filling satisfies the property $(\ast)$ above so it appears with non-zero coefficient (namely $1$) in $m_\lambda$.  Notice that this filling is standard.  We will show that the content of this filling is unique among fillings appearing with non-zero coefficient in $\Delta^\lambda(\underline{e})$, so after straightening the other fillings, this distinguished filling cannot be canceled.  Indeed, the content of this filling is $\kappa(t^\circ_\lambda)_{x,s}=([\lambda_{x,s}]_1, [\lambda_{x,s}]_2,\dotsc)$ if $(x,s)$ is the distinguished pair, and $(\kappa(t^\circ_\lambda)_{x,s})_{\beta_x-j+1}=[\lambda_{x,s}]_j$ otherwise.  This content uniquely determines the filling $t_\lambda^\circ$, so indeed $\hat\delta_\lambda(t_\lambda^\circ)$ appears with non-zero coefficient in $\overline{m}_\lambda$.  
\end{proof}

\begin{theorem}
The ring of semi-invariants $\SI_{Q, c}(\beta, r)$ is isomorphic to the semigroup ring $\C[\Lambda_{SI}(Q, c,\beta, r)]$.  
\end{theorem}
\begin{proof}
We have already shown that there is a (vector space) homomorphism $m: \C[\Lambda_{SI}(Q, c, \beta, r)]\rightarrow \SI_{Q, c}(\beta, r)$ where $m(\lambda)=m_\lambda$. 
\begin{itemize}
\item[Claim 1:] $m$ is injective.  

Suppose that $y=m(\sum_{\lambda \in T} a_\lambda \lambda) =\sum_{\lambda \in T} a_\lambda m_\lambda=0 \in \SI_{Q, c}(\beta, r)$, where $T$ is a finite subset of $\Lambda_{SI}$.  Let $\operatorname{max}(T)$ be the set of maximal elements in $T$ under the partial order $\preceq$ defined on $\Lambda$.  Then $y\in \sum\limits_{\lambda \in \operatorname{max}(T)} \mathcal{F}_{\lambda}$  Now for each $\mu \in \operatorname{max}(T)$ there is a surjection
\[
	\varphi_\mu: \sum\limits_{\lambda \in \operatorname{max}(T)} \mathcal{F}_\lambda \rightarrow M_\mu
	\]
given by the quotient of this space by the subspace $\mathcal{F}_{\prec \mu} + \sum\limits_{\lambda\in \operatorname{max}(T)\setminus \mu} \mathcal{F}_\lambda$.  Given that $y$ is a semi-invariant, its image under this map is $a_\mu \overline{m}_\mu$, since the space of semi-invariants in $M_\lambda$ is one dimensional.  By assumption, this is $0$, and since $\overline{m}_\mu \neq 0$, we must have that $a_\mu =0$ for all $\mu \in \operatorname{max}(T)$, contradicting the choice of $\operatorname{max}(T)$.  
\item[Claim 2:] The map $m$ is surjective.  

This fact exploits the same methods as the previous claim: we show that the maximal $\lambda$ appearing in a semi-invariant must be elements of $\Lambda_{SI}$, and subtract the corresponding semi-invariant $m_\lambda$ and are left with a semi-invariant function with smaller terms.  Suppose that $y\in \SI_{Q, c}(\beta, r)$, and write $y=\sum_{\lambda \in T} a_\lambda x_\lambda$ where $T\subset \Lambda$ is a finite subset (recall that $\C[\Rep_{Q, c}(\beta, r)$ has a basis given by standard fillings of all $\lambda\in \Lambda$, and take $x_\lambda$ to be the summands corresponding to $\lambda$).  Let $\operatorname{max}(T)$ again be the maximal elements in $T$ under the partial order $\preceq$.  Notice that the collection of empty partitions is indeed an element of $\Lambda_{SI}$, so we will proceed by induction on $\operatorname{height}(T)$ defined to be the length of the longest chain joining both the empty partition and an element of $\operatorname{max}(T)$.  For $\operatorname{height}(T)=0$, $m$ is a constant, which is the image of the same constant under the map $m$.  For $\mu\in \operatorname{max}(T)$, notice that $\varphi_\mu(y)=a_\mu \overline{x}_\mu$ must be a semi-invariant in $\gr(\Rep_{Q, c}(\beta, r))$, so $\mu\in \Lambda_{SI}$.  Therefore, for each $\mu \in \operatorname{max}(T)$, $a_\mu \overline{x}_\mu = b_\mu \overline{m}_\mu$.  In particular, $a_\mu x_\mu - b_\mu m_\mu \in \mathcal{F}_{\prec \mu}$.  Now let 
\[
	y_1 = y - \sum_{\mu \in \operatorname{max}(T)} b_\mu m_\mu.
	\]
By the above remarks, then, $y_1 = \sum_{\lambda \in T_1} a'_\lambda x_\lambda$ where $T_1=\{\lambda \prec \operatorname{max}(T)\}$.  As the difference of semi-invariants, $y_1$ is itself a semi-invariant, and $\operatorname{height}(T_1)<\operatorname{height}(T)$.  By induction, then, $y_1=\sum_{\lambda \in \Lambda_{SI}} b_\lambda m_\lambda$, so 
\[
	y=\left(\sum_{\mu \in \operatorname{max}(T)} b_\mu m_\mu\right) + \left(\sum_{\lambda\in T_1} b_\lambda m_\lambda\right).
	\]
\item[Claim 3:] $m$ is a semigroup homomorphism.  

This is proven directly.  It has already been shown that $m_\lambda \cdot m_\mu \in \mathcal{F}_{\lambda+\mu}$.  Now $\Delta^\lambda(\underline{e})$ is a linear combination of all multitableau of shape $\lambda$ such that $I(t_\lambda)_{x,s,i}\cup I(t_\lambda)_{x,s', \sigma(\lambda)_x-i+1} = \{1,\dotsc, \beta_x\}$.  The coefficient of each multitableau is the sign of the permutation taking the sequence $(I(t_\lambda)_{x,s,i}, I(t_\lambda)_{x,s', \sigma(\lambda)_x-i+1})$ into increasing order.  Now consider $\Delta^{\lambda+\mu}(\underline{e})$.  We will simply show a bijection between pairs $t_\lambda, t_\mu$, summands in $\Delta^\lambda(\underline{e})$ and $\Delta^\mu(\underline{e})$, respectively, and summands in $\Delta^{\lambda+\mu}(\underline{e})$, and show that the signs agree.  To this end, consider $[(\lambda+\mu)_{x,s}]$.  Recall that this is the shape given by adjoining $(\beta_x-(\lambda+\mu)(i(x,s))$ and $(\lambda+\mu)(o(x,s))$.  Notice that by definition of $(\lambda+\mu)(o(x,s))$, we can choose indices $1\leq i_1< i_2 <\dotsc <i_{\lambda(o(x,s))_1}\leq (\lambda+\mu)(o(x,s))_1$ such that 
\begin{align*}
((\lambda+\mu)(o(x,s))'_{i_1}, (\lambda+\mu)(o(x,s))'_{i_2}, \dotsc, (\lambda+\mu)(o(x,s))'_{i_{\lambda(o(x,s))_1}})\\
=(\lambda(o(x,s))_1, \dotsc, \lambda(o(x,s))_{\lambda(o(x,s))_1}).
\end{align*}
 This is easiest to see in a picture:
\begin{align*}
\begin{ytableau}
*(gray) & *(gray) & *(gray) & *(gray)\\
*(gray) & *(gray) & *(gray) & \none\\
*(gray) \\
*(gray) \end{ytableau} + \begin{ytableau} *(white)& *(white) & *(white)\\ *(white)&*(white)\\ *(white)\end{ytableau} = \begin{ytableau} *(gray) & *(white) & *(white)& *(gray) & *(gray) & *(gray) & *(white)\\ *(gray) & *(white) & *(white) & *(gray) & *(gray) \\ *(gray) & *(white)\\ *(gray)\end{ytableau}
\end{align*}
In fact, the entire shape $[(\lambda+\mu)_{x,s}]$ can be partitioned into columns in such a way that the gray columns constitute $[\lambda_{x,s}]$ and those in white constitute $[\mu_{x,s}]$.  Now for each distinguished pair $(x,s)\in \X$, choose such a partition of the columns, and partition the columns of the other shapes $[\lambda_{x,s'}]$ accordingly, namely if the column $i$ of $[(\lambda+\mu)_{x,s}]$ is colored gray, then the $\sigma(\lambda+\mu)-i+1$ column of $[(\lambda+\mu)_{x,s'}]$ is colored gray as well.  Fixing this partition of the columns, we have that a multitableau of shape $(\lambda+\mu)$ gives rise uniquely to a multitableau of shape $\lambda$ (given by gray columns), and a multitableau of shape $\mu$, and every pair of multitableau of shapes $\lambda$ and $\mu$ determine a filling of $(\lambda+\mu)$ by the same partitioning of the columns.  So indeed $\Delta^{\lambda+\mu}(\underline{e})$ consists of a linear combinations of all products of pairs of multitableau of shapes $\lambda$ and $\mu$.  Furthermore, since the sign is calculated by taking the product of the signs given by reordering columns, it is evident that the sign of the product agrees with the sign in $\Delta^{\lambda+\mu}(\underline{e})$.  
\end{itemize}
\end{proof}


\section{Combinatorics: The Semigroup $\Lambda_{SI}(Q, c, \beta, r)$}\label{sec:COMBINATORICS}
In this section, we determine the structure of the semigroup $\Lambda_{SI}$.  As we have shown above, $$\SI_{Q, c}(\beta, r) \cong \C[\Lambda_{SI}(Q, c, \beta, r)].$$  We will exhibit a grading on $\C[\Lambda_{SI}]$, and show that $\C[\Lambda_{SI}]$ is a polynomial ring over a sub-semigroup ring which we denote by $\C[U]$.  For this section, we fix a quiver $Q$, a coloring $c$, a dimension vector $\beta$, and a maximal rank sequence $r$.  For ease of presentation we will use $\Lambda=\Lambda(Q, c, \beta, r)$ and $\Lambda_{SI}$ similarly.    

\begin{definition} Let $\{\alpha_i^x\}_{\substack{x\in Q_0\phantom{kkkkk}\\ i=1,\dotsc, \beta_x-1}}$ be the simple roots for the group $\SL(\beta)$.  I.e., for $\lambda\in \Lambda$, $\alpha_x^i(\lambda(x,s)):=\lambda(x,s)_i - \lambda(x,s)_{i+1}$
\end{definition}
\begin{proposition}\label{flip-involution}
We have that $\lambda\in \Lambda_{SI}$ if and only if both of the following hold:
\begin{itemize}
\item For every coupled vertex $x$ with $(x,s_1), (x,s_2)\in \X$, and every $i=1,\dotsc \beta_x-1,$ $$\alpha^{x}_i (\lambda(x, s_1)) = \alpha^{x}_{\beta_x-i}(\lambda(x, s_2));$$
\item For every lonely vertex $x$, say $(x, s)\in \X$, $$\alpha^x_i(\lambda(x, s)) = 0.$$
\end{itemize}
\end{proposition}
\begin{proof}
Indeed, the equality in the proposition holds if and only if 
\begin{align*}
\lambda(x,s_1(x))_i - \lambda(x, s_1(x))_{i+1} &= \lambda(x, s_2(x))_{\beta_x-i} - \lambda(x, s_2(x))_{\beta_x-i+1}\\
\Leftrightarrow \lambda(x,s_1(x))_i + \lambda(x,s_2(x))_{\beta_x-i} &= \lambda(x, s_1(x))_{i+1} + \lambda(x, s_2(x))_{\beta_x-i+1}\\
\Leftrightarrow \lambda(x, s_1(x))_i + \lambda(x,s_2(x))_{\beta_x-i} &= \lambda(x, s_1(x))_j + \lambda(x,s_2(x))_{\beta_x-j}:=\sigma_x
\end{align*}
This is precisely the set of conditions given by proposition \ref{prop:cauchy}.
 \end{proof}
 
To organize the equations that arise from proposition \ref{flip-involution}, we will set up some notation and define a graph whose vertices are simple roots, with multiplicity.  
 \begin{definition}\label{PEG:definition}\ 
\begin{itemize}
\item[a.] Denote by $\Sigma=\Sigma(Q, c, \beta)$ the set of labeled simple roots $\{\alpha_i^{(x,s)} \mid (x,s)\in \X,\ i=1,\dotsc, \beta_x-1\}$ (namely the simple roots from above but with multiplicity for the colors included).
\item[b.] For each $\lambda \in \Lambda$, define the function $f_\lambda: \Sigma\rightarrow \N$ by $$f_\lambda(\alpha_i^{(x,s)}) :=\alpha_i^{(x,s)}(\lambda(x,s)).$$
\item[c.] Define the \emph{partition equivalence graph}, written $\PEG(Q, c, \beta, r)$ to be the graph with vertices given by the set $\Sigma$ and the following edges:
\begin{itemize}
\item[i.] for each coupled vertex $x\in Q_0$, with associated pair $(x,s_1), (x,s_2)\in \X$ say, and each $i=1,\dotsc, \beta_x-1$, define an edge $\xymatrix{\alpha_i^{(x,s_1)}\ar@{-}[r] & \alpha_{\beta_x-i}^{(x,s_2)}}.$
\item[ii.] for each arrow $a: x \rightarrow y$, and each $i=1,\dotsc, r(a)-1$, define an edge $$\xymatrix{\alpha_i^{(x,s)} \ar@{-}[r] & \alpha_{\beta_y-i}^{(y,s)}}.$$
\end{itemize}
\end{itemize}
\end{definition}
In words, edges of the first type connect labeled simple roots arising from the same $\SL(\beta_x)$, i.e., from the same vertex, and edges of the second type connect simple roots along colors.  For this reason we may call edges of the second type colored edges.    
\begin{proposition}\label{PEG:proposition}
Let $\lambda \in \Lambda$.  Then $\lambda\in \Lambda_{SI}$ if and only if $f_\lambda(\alpha) = f_\lambda(\alpha')$ whenever $\alpha$ and $\alpha'$ are in the same connected component of the PEG and $f_\lambda(\alpha) = 0$ if $\alpha$ corresponds to a root at a lonely vertex.
\end{proposition}
\begin{proof} 
Let $a\in Q_1$ be an arrow of color $s$ with $ta=x$ and $ha=y$.  Then $\lambda\in \Lambda$ implies that $f_\lambda\left(\alpha_i^{(x,s)}\right) =f_\lambda\left(\alpha_{\beta_y-i}^{(y,s)}\right)$, i.e., $f_\lambda(\alpha)=f_\lambda(\alpha')$ whenever $\alpha, \alpha'$ are connected by a colored edge.  This is so because if $\lambda \in \Lambda$, then 
\begin{align*}
f_\lambda\left(\alpha_i^{(x, s)}\right) &= \lambda(x, s)_i - \lambda(x, s)_{i+1}\\ &= \lambda(a)_i - \lambda(a)_{i+1}\\ &= (-\lambda(a)_{i+1}) - (-\lambda(a)_i) \\ &= \lambda(y, s)_{\beta_y-i} - \lambda(y,s)_{\beta_y-i+1}\\
	&= f_\lambda\left(\alpha_{\beta_y-i}^{(y,s)}\right).
	\end{align*}
But proposition \ref{flip-involution} shows that $M_\lambda$ contains a semi-invariant if and only if $f_\lambda(\alpha)=f_\lambda(\alpha')$ whenever $\alpha, \alpha'$ are linked by an edge of type (i).  Therefore, $\lambda\in \Lambda_{SI}$ if and only if equality holds for all roots in the same connected component. 
\end{proof}

\begin{proposition}\label{prop:KGrading}
Let $K_1,\dotsc, K_l$ be the list of connected components in $\PEG(Q, c, \beta, r)$, and let $\{\alpha(i)\}_{i=1,\dotsc, l}$ be some set of elements in $\Sigma$ such that the vertex corresponding to $\alpha(i)$ is in the component $K_i$ for each $i$.  For any vector $g=(g_1,\dotsc, g_l) \in \N^l$, let $V_g$ be the vector space with basis $\{x_\lambda\mid \lambda \in \Lambda_{SI},\ f_\lambda \alpha(i) = g_i\}$.  Then $$\C[\Lambda_{SI}] = \bigoplus\limits_{g\in \N^K} V_g$$ is a graded direct sum decomposition of the semigroup ring $\C[\Lambda_{SI}]$.  In other words, $\C[\Lambda_{SI}]$ has a multigrading by the connected components of $\PEG(Q, c, \beta, r)$.  
\end{proposition}
This follows immediately from the description of the semigroup structure of $\Lambda_{SI}$ above and proposition \ref{PEG:proposition}.  

\begin{definition}
Let $E=E_{Q, c}(\beta, r)$ be the set of elements in $\Sigma$ whose corresponding vertices are endpoints for the PEG associated to $(Q, c, \beta, r)$.  For an element $e\in E$ which is contained in the string, write $\Theta(e)$ for the distinct second endpoint contained in this string (we do not consider an isolated vertex to be a string).  Clearly $\Theta:E \rightarrow E$ is an involution.
\end{definition}
In fact, we can explicitly describe $E$.
\begin{proposition}
Each endpoint of the PEG is of one of the following two mutually exclusive forms:
\begin{itemize}
\item[I.] if $x$ is coupled and $(x, s)\in \X$, then $\alpha_i^{(x, s)}$ is an endpoint for $r(o(x,s))\leq i \leq \beta_x-r(i(x,s))$;
\item[II.] if $x$ is lonely and $(x, s)\in \X$, then $\alpha_i^{(x, s)}$ is an endpoint for $1\leq i \leq \beta_x$.
\end{itemize}
\end{proposition}
\begin{proof}
This is a consequence of the definition \ref{PEG:definition}.  We will call the edges that connect roots on the same vertex of the quiver non-colored, and those that connect roots on different vertices of the quiver colored.  If $x$ is lonely then there can only possibly be colored edges containing any of the elements $\alpha_i^{(x, s)}$, and by definition, each vertex can be contained in at most one such.  If, however, $x$ is coupled and $(x, s)\in \X$, then each vertex $\alpha_i^{(x, s)}$ is incident to precisely one non-colored edge.  Those with $i<r(o(x,s))$ or $i>\beta_x-r(i(x,s))$ are also incident to a colored edge by definition.  For $r(o(x,s))\leq i \leq \beta_x-r(i(x,s))$, there are no colored edges incident to $\alpha^{(x,s)}_i$.
\end{proof} 
We will use the endpoints of the strings to find a system of equations so that each positive integer-valued solution of the system will correspond to an element $\lambda \in \Lambda_{SI}$.  
\begin{remark}
Below lists the endpoints in $\{\alpha^i_{x,s}\}_{i=1,\dotsc, \beta_x}$ and calculates the values of $f_\lambda$ on such endpoints.  In order to write the system of equations mentioned above in a compact form, we also label these possibilities:
\begin{itemize}
\item[a.] If $r(o(x,s)) + r(i(x,s)) = \beta_x$ for some $(x,s) \in \X$, then $\alpha_{r(o(x,s))}^{(x,s)}$ is the unique endpoint in this set.  For this endpoint, we have \begin{align*}f_\lambda\left(\alpha_{r(o(x,s))}^{(x,s)}\right) = \lambda(o(x,s))_{r(o(x,s))} + \lambda(i(x,s))_{r(i(x,s))}.\end{align*}  We will denote this endpoint by $(o(x,s), i(x,s))$.
\item[b.] If $r(o(x,s))+r(i(x,s))<\beta_x$ for some $(x,s)\in \X$, then $\alpha_{r(o(x,s))}^{(x,s)}$ is an endpoint, and \begin{align*}f_\lambda \left(\alpha_{r(o(x,s))}^{(x,s)}\right) = \lambda(o(x,s))_{r(o(x,s))}.\end{align*}  We will denote this endpoint by the arrow $o(x,s)$.
\item[c.] If $r(o(x,s))+r(i(x,s))<\beta_x$ for some $(x,s)\in \X$, then $\alpha_{\beta_x-r(i(x,s))}^{(x,s)}$ is an endpoint, and \begin{align*}f_\lambda \left(\alpha_{\beta_x-r(i(x,s))}^{(x,s)}\right) = \lambda(i(x,s))_{r(i(x,s))}.\end{align*}  Such an endpoint will be denoted by the arrow $i(x,s)$.
\item[d.] Finally, if $r(o(x,s))<i<\beta_x-r(i(x,s))$, or $(x,s)$ has no mirror and $i\neq r(o(x,s))$, $i\neq \beta_x-r(i(x,s))$, then \begin{align*}f_\lambda\left(\alpha_i^{(x,s)}\right) = 0.\end{align*}  Such endpoints will be denoted by the symbol $0_i^{(x,s)}$.  
\end{itemize}
Thus, an endpoint can be of type Ia, Ib, Ic, Id, or type IIa, IIb, IIc, IId.  
\end{remark}
We will illustrate the these possibilities by means of an example.  Recall that example \ref{examples-CSA} $\star$ was the following colored string algebra: 
\begin{align*}
\xymatrix@C=12ex{
			&			& 3 \ar@{->}[dr] && \\
1 \ar@<-.7ex>@{~>}[r] \ar@<.7ex>@{..>}[r] & 2 \ar@{~>}[ur] \ar@{..>}[rr] && 4 \ar@<.7ex>@{..>}[r] \ar@<-.7ex>@{->}[r]& 5 }
\end{align*}
Consider $\Rep_{Q, C}(\beta, r)$ with $\beta, r$ as indicated in the diagram 
\begin{align*}
\xymatrix@C=12ex{
			&			& *+[F]{2} \ar@{->}[dr]|{2} && \\
*+[F]{2}\ar@<-.7ex>@{~>}[r]|{2} \ar@<.7ex>@{..>}[r]|{2} & *+[F]{6} \ar@{~>}[ur]|{2} \ar@{..>}[rr]|{2} && *+[F]{4} \ar@<.7ex>@{..>}[r]|{2} \ar@<-.7ex>@{->}[r]|{2}& *+[F]{2} }
\end{align*}
The PEG is given below:
\begin{align*}
\includegraphics[width=17cm]{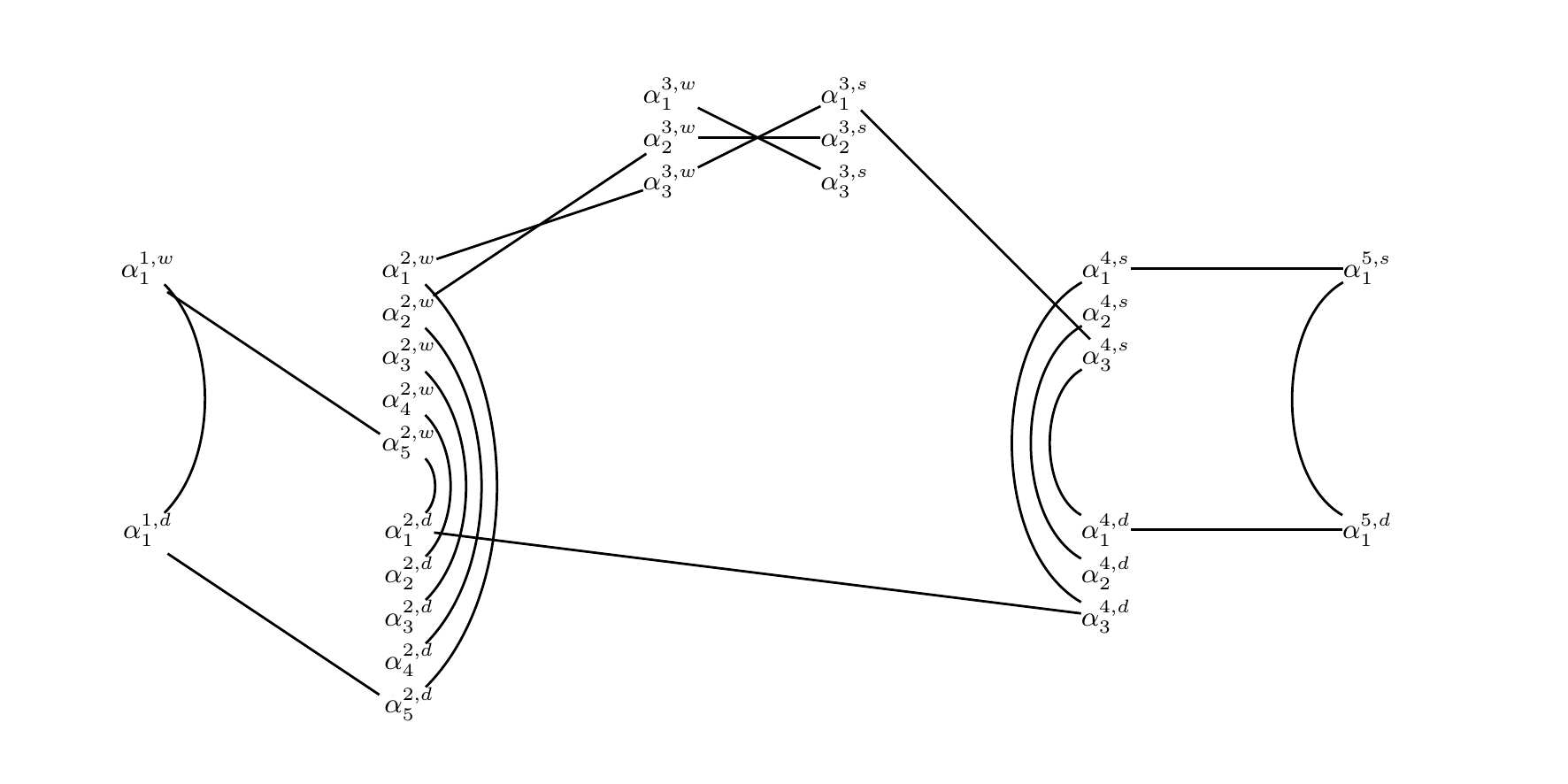}
\end{align*}

\begin{definition}
For any $\lambda\in \Lambda$, define $u_\lambda: Q_1 \rightarrow \N$ to be the function $u_\lambda(a) = \lambda(a)_{r(a)}$.  For any function $u:Q_1\rightarrow \N$, let $\varphi_u:E \rightarrow \N$ be the function defined as follows:
\begin{align*}
\varphi_u(e)&=\begin{cases} u(i(x,s)) + u(o(x,s)) & \textrm{ if the endpoint } e \textrm{ is of type (Ia) and labeled } (o(x,s),i(x,s))\\
 u(o(x,s)) & \textrm{ if the endpoint } e \textrm{ is of type (Ib) and labeled } o(x,s)\\
u(i(x,s))  & \textrm{ if the endpoint } e \textrm{ is of type (Ic) and labeled } i(x,s)\\
 0 & \textrm{ if the endpoint } e \textrm{ is of type (Id) or (II)}. \end{cases}
 \end{align*}
 We call $\varphi_u$ the \emph{companion function} to $u$.
 \end{definition}

We will denote by $U=U(Q,c,\beta,r)$ the set of functions $u:Q_1\rightarrow \N$ such that $\varphi_u(e)=\varphi_u(\Theta(e))$ for all $e\in E$.  Notice that $U$ is a semigroup with respect to the usual addition of functions.  

\begin{proposition}
If $\lambda \in \Lambda_{SI}$ then $u_\lambda\in U(Q, C, \beta, r)$.
\end{proposition} 
\begin{proof}
This is clear from proposition \ref{PEG:proposition}, together with the fact that if $\lambda \in \Lambda_{SI}$, and $x$ is a lonely vertex, $(x,s)\in \X$, then $f_\lambda(\alpha_i^{(x,s)})=0$ for $i=1,\dotsc, \beta_x$.
\end{proof}  
Notice that from $u_\lambda$ one can calculate the values of $f_\lambda(\alpha)$ whenever $\alpha\in \Sigma$ is in a string of the PEG.
\begin{definition} Denote by $Y=Y(Q, c, \beta, r)$ the set of maps $y:\{\textrm{bands in } \Sigma\} \rightarrow \N$.  For any $u\in U$ and $y\in Y$, take $\lambda_{u,y}:Q_1 \rightarrow \Pp$ to be the map defined by the following conditions:
\begin{align*}
\lambda_{u,y}(a)_{r(a)}&=u(a)\\
\alpha(\lambda_{u,y}) &= \begin{cases} \varphi_u(e) & \textrm{ if } e \textrm{ is an endpoint of the string containing } \alpha\\ y(b) & \textrm{ if } \alpha \textrm{ is contained in the band } b. \end{cases}
\end{align*} 
\end{definition}

\begin{remark}
Let us summarize the results above:
\begin{itemize}
\item[i.] The set $U$ is a semigroup with respect to the usual addition of functions,
\item[ii.] $\lambda_{u,y}(a)$ has at most $r(a)$ non-zero parts, so $\lambda_{u,y}\in \Lambda$,
\item[iii.] $\operatorname{image}( (u,y)\mapsto \lambda_{u,y} ) \subset \Lambda_{SI}$.
\end{itemize}
\end{remark}

\begin{proposition}
The map $(u,y)\mapsto \lambda_{u,y}$ is a semigroup isomorphism between $U\times Y$ and $\Lambda_{SI}$.
\end{proposition}
\begin{proof}
We construct an inverse explicitly.  For any $\lambda\in \Lambda_{SI}$, define $(u_\lambda, y_\lambda)$ as follows:
\begin{align*}
u_\lambda(a)&:= \lambda(a)_{r(a)}\\
y_\lambda(b)&:= f_\lambda (\alpha) \textrm{ for any } \alpha \textrm{ in the band } b.
\end{align*}
It is routine that $u_{\lambda(u,y)} = u$ and $y_{\lambda(u,y)} = y$, so this is indeed a bijection, and it is clear that the composition operation in $U\times Y$ is preserved under this map.
\end{proof}
\begin{corollary}
We have the following ring isomorphism $$\SI_{Q, c}(\beta, r) \cong \C[U(Q, c, \beta, r)] [y_b]_{b\in \{ \textrm{bands in } \Sigma\}},$$ that is, $\SI_{Q, c}(\beta, r)$ is a polynomial ring over the semigroup ring $\C[U]$. 
\end{corollary}

\begin{proposition}\label{prop:generalUproperties}
The semigroup $U(Q, C, \beta, r)$ is a sub-semigroup of $\N^{Q_1}$, satisfying the following:
\begin{itemize}
\item[a.] $U(Q, C, \beta, r)=\{(u_a)_{a\in Q_1} \in \N^{Q_1}\mid \varphi_{u}(e) = \varphi_u(\Theta(e)) \textrm{ for } e \in E\}$,
\item[b.] $\varphi_u(e) = \sum\limits_{a\in Q_1} c_a^e u_a$ with $c_a^e\in \{0,1\}$ for each endpoint $e\in E$,
\item[c.] $u_a$ appears with nonzero coefficient in at most two functions $\varphi_u$. I.e., for each $a\in Q_1$, there are at most two endpoints $e_1,e_2\in E$ with $c_a^{e_1}=c_a^{e_2} = 1$.
\end{itemize}
\end{proposition}
\begin{proof}
(a) is simply the definition of $U(Q, C, \beta, r)$, reformulated as a sub-semigroup of $\N^{Q_1}$, while (b) is the definition of the control equations.  Recall that $\varphi_u(e)=u(a)+u(b)$, $u(a)$ or $0$ for any endpoint $e$, and since the quiver is acyclic, $a\neq b$, so the coefficient on any summand is at most $1$.  To show (c), we recall that $r$ is a \emph{maximal} rank sequence for $\beta$.  This implies that if $e_1$ is an endpoint of type (Ia) labeled $(a,b)$ (in which case $\varphi_u(e_1) = u_b + u_a$), then the only other type of endpoint labeled with an $a$ is either another of type (Ia) labeled $(c,a)$, or one of type (Ic) labeled $a$.  (Similarly the only other type of endpoint labeled with a $b$ is either another of type (Ia) labeled $(b,c)$, or of type (Ib) labeled $b$.)
\end{proof}


\section{Matching Semigroups}\label{sec:semigroupU}
Fix $(Q, c)$ a gentle string algebra  $ \beta$ a dimension vector vector and rank sequence, together with its PEG, $\Sigma$.  We will define a general class of sub-semigroups of $\N^l$, of which all $U(Q, C, \beta, r)$ are members. We will then describe a general procedure for determining generators and relations for these semigroup rings by means of a graph, and show that the generators of these semigroups occur in multidegree at most 2.  First, however, we exhibit some structure enjoyed by $k[U]$.
\begin{theorem}
The semigroup ring $\C[U]$ is the coordinate ring of an affine toric variety. 
\end{theorem}
\begin{proof}
Let $\C[X_a]_{a\in Q_1}$ be the polynomial ring on the arrows of $Q_1$, and let $S$ be the set of strings in $\Sigma$.  Suppose that the PEG has the following endpoints: $\{e_1^{(s)}, e_2^{(s)}\}_{s\in S}$.  Then we define the action of $(\C^*)^S$ on $\C[X_a]_{a\in Q_1}$ as follows: suppose that $(t_s)_{s\in S} \in (\C^*)^S$, then
\begin{align*}
(t_s). \prod\limits_{a\in Q_1} X_a^{u(a)} := t_s^{\varphi_u(e_1^{(s)}) - \varphi_u(e_2^{(s)})} \prod\limits_{a\in Q_1} X_a^{u(a)}.
\end{align*}
A polynomial $p \in \C[X_a]_{a\in Q_1}$ is invariant with respect to this action if and only if its monomial terms are, so it suffices to assume $p$ is a monomial.  Suppose that a monomial $\prod\limits_{a\in Q_1}X_a^{u(a)}$ is invariant with respect to each $t_s$.  Then for each endpoint pair $\{e_1^{(s)}, e_2^{(s)}\}$, we have $$t_s. \prod\limits_{a\in Q_1} X_a^{u(a)} = t_s^{\varphi_u(e_1^{(s)}) -\varphi_u(e_1^{(s)})}\prod\limits_{a\in Q_1}X_a^{u(a)} =\prod\limits_{a\in Q_1} X_a^{u(a)},$$ so $\varphi_u(e_1^{(s)}) = \varphi_u(e_2^{(s)})$ for $s\in S$.  Therefore, such a monomial is invariant with respect to the action if and only if $u\in U$.  Then clearly $\C[U] = \C[X_a]^{(\C^*)^S}$ is the invariant ring with respect to this torus action.
\end{proof}
\begin{definition}\label{def:matchingsemigroup}
Let $\{f_i: \N^l \rightarrow \N\}_{i=1,\dotsc, 2m}$ be a collection of $\N$-linear functions $$f_i(x_1,\dotsc, x_l)=\sum_{j=1}^l c_i^j x_j$$ satisfying the following properties:
\begin{itemize}
\item[a.] $c_i^j \in \{0, 1\}$ for all $i=1,\dotsc, 2m$, $j=1, \dotsc, l$;
\item[b.] $c_i^j \neq c_{i+m}^j$ for $i=1,\dotsc, m$, $j=1,\dotsc, l$ (i.e., the equations $f_i(x_1,\dotsc, x_l) = f_{i+m}(x_1,\dotsc, x_l)$ are reduced);
\item[c.] for $j=1,\dotsc, l$, $\#\{i \mid c_i^j \neq 0,\  i=1,\dotsc, 2m\}\leq 2$ (i.e., each variable $x_j$ appears with non-zero coefficient in at most two functions $f_i$).
\end{itemize}
The semigroup 
\[
U(\{f_i\}_{i=1,\dotsc, 2m}):=\{\underline{u}=(u_1,\dotsc, u_l)\in \N^l \mid f_i(\underline{u})=f_{m+i}(\underline{u}),\ i=1,\dotsc,m\}
\] 
is called a \emph{matching semigroup} if the functions $f_i$ satisfy the conditions (a)-(c).
\end{definition}

The following is the main theorem of this section.
\begin{theorem}\label{thm:degreesofgenerators}
Suppose that $U=U(\underline{f})\subset \N^l$ is a matching semigroup with $\underline{f}=\{f_i\}_{i=1,\dotsc, 2m}$.  Then $U$ is generated by vectors $\underline{u}=(u_1,\dotsc, u_l)$ with the property that $f_i(\underline{u})\leq 2$ for $i=1,\dotsc, 2m$.  In particular, $u_i\leq 2$.
\end{theorem}

In order to prove this theorem, we construct a graph $G(\underline{f})$ and interpret certain walks on this graph as elements in $U$.  

\begin{definition}
Let $G(\underline{f})$ be the multigraph with two types of edges, solid and dotted, on the vertices $\{1, \dotsc, 2m\}$, with a solid edge
\begin{align*}
\xymatrix{i \ar@{-}[r] & k} & \textrm{ whenever } c_i^j=c_k^j = 1,\ i\neq k, \end{align*}
a solid loop 
\begin{align*}
\xymatrix{i \ar@{-}[r] & i } & \textrm{ whenever } i \textrm{ is the unique integer for which } c_i^j = 1,
\end{align*}
and dotted edges $\xymatrix{ i \ar@{..}[r]& m+i}$ for $i=1,\dotsc, m.$  We define a function $L:\operatorname{Edges}(G(\underline{f})) \rightarrow \{1, x_1,\dotsc, x_l\}$ with 
\begin{align*}
L(E) = \begin{cases} 1 & \textrm{ if } E \textrm{ is a dotted edge}\\ x_j & \textrm{ if } E \textrm{ is the edge containing } $i, k$ \textrm{ arising from the condition } c_i^j = c_k^j = 1. \end{cases}
\end{align*}
\end{definition}

In depicting this graph, we will indicate the labeling as a decoration on the appropriate edge.  Heuristically, each vertex $i$ stands for a function $f_i$.  A vertex $i$ is contained in a solid edge labeled $x_j$ if $x_j$ appears with non-zero coefficient in $f_i$, and the vertices corresponding to functions on either side of a defining equation of $U(\underline{f})$ are joined by a dotted edge.  The name \emph{matching semigroup} arises from the fact that the dotted edges form a perfect matching for the graph $G(\underline{f})$.  Moreover, while each vertex is contained in exactly one dotted edge, it can be contained in several solid edges: as many as non-zero coefficients in the linear function to which it corresponds.

\begin{example}\label{ex:matchingsemigroup}
We provide an example to clarify some of the more complicated definitions.  Let $U\subset \N^{11}$ be the semigroup defined by the following equations:
\begin{align*}
f_1(\underline{x}) = x_1+x_2 &= 0 = f_6(\underline{x})\\
f_2(\underline{x})= x_2+x_3 &= x_8 + x_9 =f_7(\underline{x})\\
f_3(\underline{x}) = x_3+x_4 &= x_9+x_{10} = f_8(\underline{x})\\
f_4(\underline{x})= x_5+x_4 &= x_7+x_{11} = f_9(\underline{x})\\
f_5(\underline{x})=x_6+x_7 &=x_{11} =f_{10}(\underline{x}).
\end{align*}
Then the graph $G(\underline{f})$ looks as follows:
\begin{align}\label{ex:matchinggraph}
\xymatrix@C=10ex@R=10ex{
1 \ar@[|(4)]@{-}@(l,u)^{x_1}  \ar@[|(4)]@{-}[r]_{x_2} & 2 \ar@[|(4)]@{-}[r]_{x_3} & 3  \ar@[|(4)]@{-}[r]_{x_4} & 4\ar@[|(4)]@{-}@(ur,ul)_{x_5} & 5 \ar@[|(4)]@{-}@(r,u)_{x_6}\ar@[|(4)]@{-}[dl]_{x_7}\\
6 \ar@[|(4)]@{--}[u] & 7 \ar@[|(4)]@{-}@(l,d)_{x_8} \ar@[|(4)]@{--}[u] \ar@[|(4)]@{-}[r]_{x_9} & 8\ar@[|(4)]@{-}@(dr,dl)^{x_{10}} \ar@[|(4)]@{--}[u] & 9 \ar@[|(4)]@{-}[r]_{x_{11}} \ar@[|(4)]@{--}[u]& 10 \ar@[|(4)]@{--}[u]
 }
\end{align}
\end{example}
\begin{definition}
A \emph{walk} on $G(\underline{f})$ is a sequence of vertices and edges $w=v_nE_nv_{n-1}E_{n-1}\dotsc E_1v_0$ such that $V(E_i) = \{v_i, v_{i-1}\}$ (i.e., the vertices of $E_i$ are precisely the two surrounding it in the sequence).  To each such walk, associate an integer vector $u(w) \in \N^l$ with 
\[
u(w)_j = \# \{ k \mid \textrm{ the edge } E_k \textrm{ is labeled } x_j \}.\]
\end{definition}
\begin{definition}
A walk is called \emph{alternating} if $E_k, E_{k-1}$ are of different edge types for $k\in[n]$.  An alternating walk is called
\begin{itemize}
\item[i.] a \emph{string} if both $E_1, E_n$ are loops,
\item[ii.] a \emph{band} if $v_0 = v_n$ and $E_0, E_n$ are edges of different types, and none of the $E_i$ are loops.  
\end{itemize}
\end{definition}

Henceforth, we will refer to ``alternating'' strings and bands simply as strings and bands.

\begin{example}\label{ex:alternating}
We will illustrate some strings and bands on the graph from example \ref{ex:matchingsemigroup}.  In these walks, we write $x_i$ for the edge with the given label, and $E$ for the unique dotted edge containing a given vertex.  The walk $w_1:=3E8x_97E2x_33$ is a band, and $w_2=5x_{6}5E10x_{11}9E4x_54$ is a string.  
\end{example}

\begin{lemma}\label{prop:alternatingwalk}
Suppose that $w$ is a string or band.  Then $u(w) \in U$.
\end{lemma}
\begin{example}
In example \ref{ex:alternating}, $u(w_1)_i=0$ if $i\neq 3, 6$, and $u(w_1)_3=u(w_1)_6=1$.  It is easily checked that $f_j(u(w_1))=f_{5+m}(u(w_1))$ for $i=1,\dotsc, 5$.  
\end{example}

\begin{proof}
Without loss of generality, assume $i\leq m$.  Notice that if $w$ is a string, then $$f_i(u(w))= \#\{j\in\{1,\dotsc, n-1\} \mid v_j \textrm{ is the vertex } i\},$$ while if $w$ is a band, then $$f_i(u(w))=\#\{j\in \{1,\dotsc, n\} \mid v_j \textrm{ is the vertex } i\}.$$  But $w$ is alternating, so every occurrence of the vertex $i$ is either immediately preceded or succeeded by an occurrence of the vertex $i+m$, so $$f_i(u(w)) = \#\{j \mid v_j \textrm{ is the vertex } i+m \} = f_{i+m}(u(w))$$ as claimed. 
\end{proof}

\begin{lemma}\label{lem:2cycles}
$G(\underline{f})$ contains no alternating two-cycles.
\end{lemma}
\begin{proof}
If the edge labeled $x_1$ contains two vertices $i, i+m$ which are both contained in a single dotted edge, then $f_i(\underline{x}) = x_1 + \sum c_i^j x_j = x_1 + \sum c_{i+m}^j x_j = f_{i+m}(\underline{x})$, contradicting definition \ref{def:matchingsemigroup} (b).
\end{proof}

 \begin{lemma}
 A matching semigroup $U=U(\underline{f})$ is generated by the set 
 \[ 
 	\{u(w) \mid w \textrm{ is either a string or a band on } G(\underline{f})\}.
 \]
 \end{lemma}
 \begin{proof}
Let $\leq$ be the coordinate-wise partial order on $U$.  We will show that for each $0\neq u \in U$, there is a non-trivial alternating walk $w$, which is either a string or a band, and an element $u'\in U$ such that 
 \begin{itemize}
 \item[i.] $u'\leq u$,
 \item[ii.] $u=u(w) + u'$.
 \end{itemize}

\begin{itemize}
\item[\underline{Case 1}:] Suppose that $u_{j_1} \neq 0$ for some $j_1$ for which $x_{j_1}$ is a loop.  We inductively construct a sequence of alternating walks $t_k = v_{2k} E_{2k} v_{2k-1} \dotsc v_1 E_1 v_0$ with $L(E_1)=x_{j_1}$ satisfying the following:
\begin{itemize}
\item[(1)] $0<u(t_k)<u(t_{k+1}) <u$
\item[(2)] $f_{v_{2k-1}}(u-u(t_k))+1 = f_{v_{2k}}(u-u(t_k))$
\item[(3)] $f_i(u-u(t_k)) = f_{i+m}(u-u(t_k))$ whenever $\{i,i+m\}\neq \{v_{2k}, v_{2k-1}\}$.
\end{itemize}
Let $E_1$ be the edge with $L(E_1)=x_{j_1}$, $v_0=v_1$ the unique vertex contained in this loop, $E_2$ the dotted edge containing $v_1$, and $v_2$ the unique second vertex contained in $E_2$.  
\begin{itemize}
\item[Claim 1:] $t_1$ satisfies (1)-(3).
\begin{proof}
$u(t_1)_{j_1} = 1$, so immediately $u(t_1)>0$.  Furthermore, $u(t_1)_{j'}=0$ for $j'\neq j_1$, and since $c^{j_1}_{v_1}=1$, $c^{j_1}_{v_2} = 0$, $f_{v_2}(u(t_1)) = 0$.  On the other hand, $f_{v_2}(u)=f_{v_1}(u) >0$ by assumption, so $u(t_1)<u$, and (1) is proven.  

As for (2) and (3), $f_{v_1}(u-u(t_1))=f_{v_1}(u) - 1=f_{v_2}(u)-1 = f_{v_2}(u-u(t_1))-1$ since $u\in U$.  Furthermore, if $\{i, i+m\}\neq \{v_{2}, v_{1}\}$, then $c^j_i = c^j_{i+m}=0$ since $f_{v_1}$ is the unique function in which $x_j$ appears with non-zero coefficient (as $E_1$ is a loop).  Therefore, $f_{i}(u-u(t_1))=f_{i}(u)=f_{i+m}(u) = f_{i+m}(u-u(t_1))$, proving (3).
\end{proof}
\item[Claim 2:] If $t_k= v_{2k} E_{2k} v_{2k-1} \dotsc v_1 E_1 v_0$ satisfies (1)-(3), and there is no $E_s$ for $s=2, \dotsc, 2k$ with $E_s$ a loop, then there are two possibilities:
\begin{itemize}
\item[a.] There is a loop $E_{2k+1}$ containing the vertex $v_{2k}$ such that the walk $w:=v_{2k}E_{2k+1} t_k$ is an alternating string and $u(w)\leq u$;
\item[b.] There is a solid edge $E_{2k+1}$ which is not a loop such that $t_{k+1}=v_{2k+2}E_{2k+2}v_{2k+1}E_{2k+1} t_k$ is an alternating walk satisfying (1)-(3).
\end{itemize}
Before proving this dichotomy, we note that this proves the following: if $u\in U$ such that $u_j \neq 0$ with $x_j$ a loop, then there is a an alternating string such that $u-u(w)\in U$.  Indeed, $u(t_k)<u(t_{k+1})<u$ by (1), so there must be a $t_k$ such that $u(t_k)<u$ and for which there is a loop $E_{2k+1}$ such that $w$ as defined in (a) is an alternating string and $u(w)\leq u$.  
\begin{proof}
Suppose that $t_k=v_{2k}E_{2k}v_{2k-1} \dotsc v_1 E_1 v_0$ contains no loops other than $E_1$, satisfies (1)-(3), and does not satisfy (a).  By property (2), 
\[
f_{v_{2k-1}}(u-u(t_k))+1 = f_{v_{2k}}(u-u(t_k)) = \sum\limits_{j \mid c^j_{v_{2k}}\neq 0} u_j - u(t_k)_j.
\]
Since $f_{v_{2k-1}}(u-u(t_k))\geq 0$, there must be a $j_k$ such that $u_{j_k}>u(t_k)_{j_k}$ and $c^{j_k}_{v_{2k}} = 1$.  In terms of the graph, then, there is a solid edge $E_{2k+1}$ (which is not a loop since $t_k$ does not satisfy (a)) with $L(E_{2k+1}) = x_{j_k}$ containing the vertex $v_{2k}$.  Let $v_{2k+1}$ be the distinct second vertex contained in $E_{2k+1}$, $E_{2k+2}$ the unique dotted edge containing $v_{2k+1}$, and $v_{2k+2}$ the distinct second vertex contained in $E_{2k+2}$.  Let $t_{k+1}=v_{2k+2}E_{2k+2}v_{2k+1}E_{2k+1}t_k$.  We claim that $t_{k+1}$ satisfies (1)-(3). 

Notice that $u(t_k)_{j_k}+1 = u(t_{k+1})$ (as $t_{k+1}$ has an additional occurrence of the edge labeled $x_{j_k}$), and $u(t_k)_{j'} = u(t_{k+1})_{j'}$ for $j'\neq j_k$.  Therefore, $0<u(t_k)<u(t_{k+1})$ and $u(t_{k+1})_{j'}\leq u_{j'}$ for $j'\neq j_k$.  Furthermore, $u_{j_k}>u(t_k)_{j_k}$ from above, so $u_{j_k}\geq u(t_k)_{j_k} +1 = u(t_{k+1})_{j_k}$, so $u(t_{k+1})\leq u$.  We will show in the course of proving (2) that $u(t_{k+1}) \notin U$, implying we cannot have equality, so $u(t_{k+1})<u$ as claimed.

Recall that since $v_{2k+1}$ contains the edge labeled $x_{j_k}$, $c^{j_k}_{v_{2k+1}} = 1$.  By lemma \ref{lem:2cycles}, then, $c^{j_k}_{v_{2k+2}} = 0$.  Furthermore, $f_{v_{2k+1}}(u-u(t_k))=f_{v_{2k+2}}(u-u(t_k))$ since $t_k$ satisfies condition (3).  Therefore 
\begin{align*}
f_{v_{2k+1}}(u-u(t_{k+1})) &= f_{v_{2k+1}}(u-u(t_k))-1 \\
				&= f_{v_{2k+2}}(u-u(t_k))-1\\
				&= f_{v_{2k+2}}(u-u(t_{k+1}))-1
				\end{align*}
proving (2).

Finally, if $\{i, i+m\}=\{v_{2k-1}, v_{2k}\}$, then 
\begin{align*}
f_{v_{2k-1}}(u-u(t_{k+1})) &= f_{v_{2k-1}}(u-u(t_k)) \\ 
					&=f_{v_{2k}}(u-u(t_k))-1\\
					&=f_{v_{2k}}(u-u(t_{k+1})),
					\end{align*}
while if $\{i, i+m\} \not\subset\{v_{2k-1}, v_{2k}, v_{2k+1}, v_{2k+2}\}$, then $f_i(u-u(t_{k+1})) = f_i(u-u(t_k)) = f_{i+m}(u-u(t_k)) = f_{i+m}(u-u(t_{k+1}))$, proving (3).
\end{proof}
\end{itemize}
\item[\underline{Case 2}:] Now suppose for all $j$ such that $x_j$ is a loop, we have that $u_j = 0$.  Take $j_1$ with $u_{j_1} \neq 0$ (possible since $u\neq 0$).  Let $v_0, v_1$ be the vertices (taken in some order) contained in the edge labeled $x_j$, $E_1$ this edge, $E_2$ the dotted edge containing $v_1$ and $v_2$ the other end of this edge.  Call this walk $t_1$.  Notice that $v_2\neq v_0$ by lemma \ref{lem:2cycles}.  We can again recursively define alternating walks $t_k$ starting with $t_1$ satisfying the following: if $v_0 \neq v_{2k}$, then 
\begin{itemize}
\item[(1)] $0<u(t_k)<u(t_{k+1})\leq u$
\item[(2)] $f_{v_{2k-1}}(u-u(t_k))+1 = f_{v_{2k}}(u-u(t_k))$
\item[(3)] $f_i(u-u(t_k))=f_{i+m}(u-u(t_k))$ whenever $\{i, i+m\} \not\subset\{v_{2k}, v_{2k-1}, v_0\}$.
\item[(4)] $t_k$ can be extended to an alternating walk $t_{k+1}$ which is either an alternating band with $u(t_{k+1})\leq u$ or $t_k$ satisfies (1)-(3).
\end{itemize} 
Thus, completely analogously to Case 1, there must be a $t_k$ that is a band.  As the proof is nearly verbatim of the proof of Case 1, we omit it.
\end{itemize}
Therefore, $u=\sum u(w_i)$ for $w_i$ some strings or bands.  
\end{proof}
 
 Notice that it is possible that $f_i = 0$ for some index $i\leq m$ (say), while $f_{i+m}=\sum c^j_{i+m} x_j$ with some $c^j_{i+m} \neq 0$ for some $j$.  It may not be clear why if $w$ is alternating string or band, then $u(w)_j =0$, which would be required if $u(w)\in U$.  However, if $f_i=0$, then there are no solid edges containing the vertex $i$.  Any alternating path passing through the solid edge labeled $x_j$ would then pass through the dotted edge between $i+m$ and $i$.  Since the walk couldn't finish at that vertex, it would immediately pass back through the dotted edge, contradicting the alternating property of the walk.

\begin{definition}
A string or band $w$ is called irreducible if there does not exist a pair of non-trivial strings or bands $w',w''$ satisfying $u(w)=u(w')+u(w'')$.
\end{definition}
Clearly $U$ is generated by $\{u(w)\mid w \textrm{ is an irreducible alternating string or band}\}$.  
\begin{lemma}\label{lem:degreeboundforwalks}
If $w$ is an irreducible string or band, then $f_i(u(w)) \leq 2$ for $i=1,\dotsc, 2m$.  
\end{lemma}
\begin{proof}
Suppose that $w=v_n E_n \dotsc E_1 v_0$ is an irreducible string or band, and $f_i(u(w))\geq 3$ for some $i=1,\dotsc, m$ (in particular, $f_{i+m}(u(w))\geq 3$).  This implies that the vertex $i$ appears in the set $\{v_1,\dotsc, v_{n-1}\}$ at least thrice.  Let $E$ be the dotted edge containing the vertices $i$ and $i+m$.  Recall that in an alternating path, each occurrence of the vertex $i$ is immediately succeeded or immediately preceded by an occurrence of $i+m$.  Let $1\leq k_1<k_2<k_3\leq n-1$ be the first three integers such that $v_{k_j} = i$, and $1\leq l_1<l_2<l_3\leq n-1$ the first three such that $v_{l_j}= i+m$.  Suppose without loss of generality that $k_1<l_1$. We claim that if $k_2<l_2$ or $l_3<k_3$, then $w$ is not irreducible.  In this case, $k_2<l_2$ implies that $w$ contains a sub-band, namely $$w=\dotsc v_{l_2} E (v_{k_2} E_{k_2} \dotsc v_{l_1} E v_{k_1}) \dotsc$$  In a diagram (although the graph is undirected, the sequence of edges and vertices of the walk will be indicated with arrows): 
\begin{align*}
\includegraphics{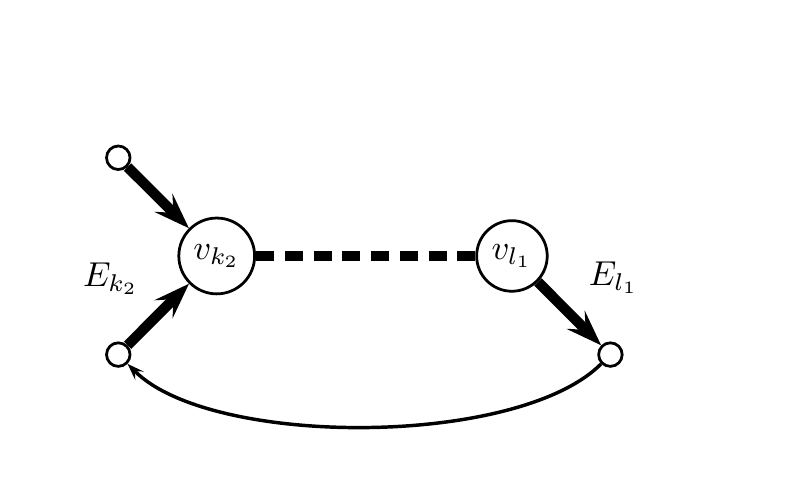}
\end{align*}
(Here the thinner arc connecting the two bottom vertices represents an alternating walk that starts and ends with dotted edges.)  This contradicts the assumption of irreducibility, so $k_2>l_2$, and the same contradiction implies that $k_3<l_2$, so we have that  $k_1<l_1<l_2<k_2<k_3<l_3$.  But now we have that $$w=\dotsc v_{l_3} (E v_{k_3} E_{k_3} \dotsc  E_{k_2+1}v_{k_2} E v_{l_2} E_{l_2} \dotsc E_{l_1+1} v_{l_1})E v_{k_1}\dotsc$$ which contains the parenthesized band.  In diagram form:
\begin{align*}
\includegraphics{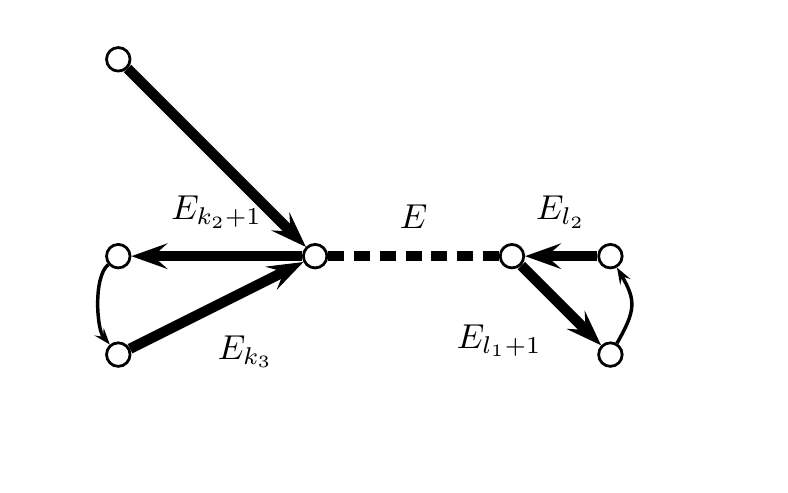}
\end{align*}
again contradicting irreducibility of $w$. 
\end{proof} 

\subsection*{proof of theorem \ref{thm:degreesofgenerators}} $U(\underline{f})$ is generated by the $u(w)$ for $w$ irreducible strings and bands, and for such walks, $f_i(u(w))\leq 2$ for $i=1,\dotsc, 2m$ by lemma \ref{lem:degreeboundforwalks}.  This concludes the proof.
\begin{flushright} $\boxtimes$ \end{flushright}

The presentation of $U(\underline{f})$ using walks on a graph allows us to determine the relations in the ring $\C[U(\underline{f})]$ as well.  Let $W(\underline{f})$ be the free semigroup generated by the irreducible paths $w_i$ on $G(\underline{f})$, and extend the function $u$ to $W(\underline{f})$ linearly.  Let $\sim_W$ be the kernel equivalence of this map, i.e., $A\sim_W B$ if and only if $u(A)=u(B)$.  The relation $\sim_W$ is a semigroup congruence, so $W(\underline{f})/\sim_W$ is a semigroup isomorphic to $U(\underline{f})$, and $\C[U(\underline{f})]$ is isomorphic to $\C[W(\underline{f})]/I_W$ where $I_W$ is generated by all elements $t_w -t_{w'}$ for $w\sim_W w'$.

\begin{remark}
Notice that since $\sim_W$ is a semigroup congruence, one has cancellation.  That is $a+b \sim_W a+c$ if and only if $b\sim_W c$.  This can  be recognized immediately from the definition of $\sim_W$.  
\end{remark}

\begin{definition}\ 
\begin{figure}[h]
\centering
\subfloat[$X$-Configuration about $E$]{\label{fig:XRelations}\includegraphics[width=3in, height=2.2in]{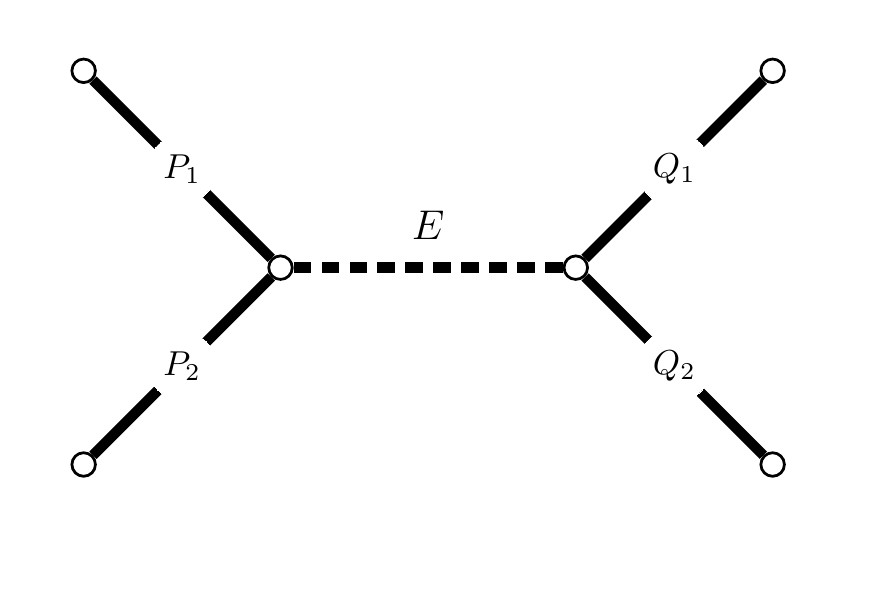}}
\subfloat[$H$-Configuration about $E, E'$]{\label{fig:HRelations}\includegraphics{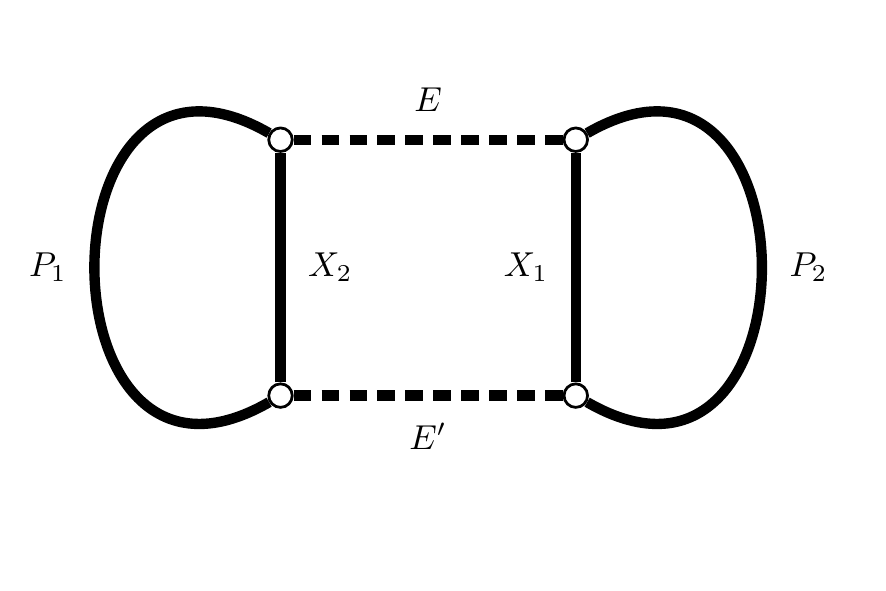}}
\caption{Relations in Graphical Form}\label{fig:Relations}
\end{figure}
\begin{itemize}
\item A walk $P$ is called a \emph{partial string} if its first edge is a loop and its last edge is solid;
\item Suppose that $P_1, Q_1$ are partial strings as in the configuration of figure \ref{fig:XRelations}.  We will often abbreviate by $Q_1P_1$ the alternating string obtained by joining $Q_1$ and $P_1$ by the edge $E$.
\item Suppose that $P_1, X_1$ are alternating walks as in figure \ref{fig:HRelations}.  Then we write $X_1P_1$ for the alternating band obtained by joining $P_1$ and $X_1$ along the edges $E$ and $E'$.  
\item Let $\sim_X$ be the minimal semigroup equivalence containing the relations:
\begin{itemize}
\item[i.] $Q_1P_1 + Q_2P_2 \sim Q_2P_1 + Q_1P_2$ for every collection  $P_1, P_2, Q_1, Q_2$ of partial strings in an $X$-configuration (figure \ref{fig:XRelations}) on $G(\underline{f})$;
\item[ii.] $X_1P_1 + X_2P_2 \sim X_1X_2 + P_1P_2$ for every collection of alternating walks $X_1, X_2, P_1, P_2$, none containing loops, in an $H$-configuration (figure \ref{fig:HRelations}) on $G(\underline{f})$.
\end{itemize}
\end{itemize}
\end{definition}

\begin{remark}
Notice that for a given pair $P, Q$ of partial strings as in figure \ref{fig:XRelations} (or a pair of alternating walks $X_1, P_1$ as in figure \ref{fig:HRelations}), $QP$ (resp. $X_1P_1$) may not be irreducible even while $Q, P$ (resp. $X_1, P_1$) contain no sub-bands.
\end{remark}

\begin{proposition}
The equivalence relations $\sim_W$ and $\sim_X$ coincide.
\end{proposition}
\begin{proof}
Notice that if two elements are equivalent under $\sim_X$, then they are equivalent under $\sim_W$, as can be seen on the relations that generate the semigroup.  

The converse is proven by induction.  Suppose that $A\sim_W B$ for some $A, B\in W(\underline{f})$.  We will show that $A\sim_X B$.  Notice that the function $u: W(\underline{f}) \rightarrow U(\underline{f})$ induces a partial order on $W(\underline{f})$ via $A' \preceq A$ if and only if $u(A')\leq u(A)$.  Notice that for any $A$, the set $\{0\preceq A' \preceq A\}$ is finite, so we can induct on $u(A)$.

 For $u(A)=0$, the proposition is clear: $u(A)=0$ implies $u(B)=0$, so $A=B=0$, which are trivially equivalent under $\sim_X$.  Now suppose that the implication holds for all $A'\prec A$.  We can assume, without loss of generality, that $a_0 \neq 0$ while $b_0=0$, since otherwise cancellation would allow us to express the equivalence under $\sim_W$ for $A'\prec A$, which, by induction, would imply equivalence under $\sim_X$.  We state the following lemma and delay the proof in order to show that the proposition follows from it.
 
 \begin{lemma}\label{lem:untwisting} With all of the above assumptions, $B\sim_X w_0 + B'$ for some $B'\in W(\underline{f})$.  \end{lemma}
 Assuming that the claim holds, then by the first paragraph of the proof, $B\sim_W w_0 + B'$.  By transitivity, then $A= w_0 + A' \sim_W w_0 + B'$.  But $\sim_W$ is a semigroup congruence, so the aforementioned equivalence holds if and only if $A'\sim_W B'$.  By inductive hypothesis, then, $A' \sim_X B'$.  Therefore, $A=w_0 + A' \sim_X w_0 + B' \sim_X B$ as desired.  
\begin{proof}[proof of lemma \ref{lem:untwisting}]
For two strings $w, w'$, choose a longest partial string common to both $w, w'$, and denote it by $(w||w')$.  (This may not be unique, but we simply choose one such for each pair of strings.) Let $l(w||w')$ be the length of this partial string (notice that $l(w||w')$ is odd since the first and last edges are solid and the walk is alternating).

\begin{itemize}
\item[Case 1:] Suppose that $w_0$ is a string.  Let $j$ be an index such that $u(w_0)_j > 0$ and $x_j$ is a loop.  Since $u(A)_j >0$ and $A\sim_W B$, we must have that $u(B)_j >0$, so there exists a string $w_{i_1}$ such that $u(w_{i_1})_j >0$, $b_{i_1}\neq 0$, and such that $l(w_0||w_{i_1})$ is maximal.  We show the following: if $w_{i_1} \neq w_0$, then $B\sim_X \Phi(B)$ in such a way that there is a walk $w_{i_2}$ appearing with non-zero coefficient in $\Phi(B)$ such that $l(w_0||w_{i_2})>l(w_0||w_{i_1})$.  Since the length of $w_0$ if finite, there must be an $N>0$ such that $w_0$ appears with non-zero coefficient in $\Phi^N(B)$.  Since equivalence under $\sim_X$ implies equivalence under $\sim_W$, then, we have that $A\sim_W \Phi^N(B)$, so $A\sim_W w_0+B'$ for some $B'$, as desired.

Let $v$ be the last vertex in $(w_0||w_{i_1})$, $E$ the dotted edge containing said vertex, $v'$ the other vertex contained in $E$, and $Q$ the partial string such that $Q(w_0||w_{i_1}) = w_{i_1}$.  This is demonstrated in the diagram below, where the walk $w_0$ is depicted in black, and $w_{i_1}$ is in gray:
\begin{align*}
\includegraphics{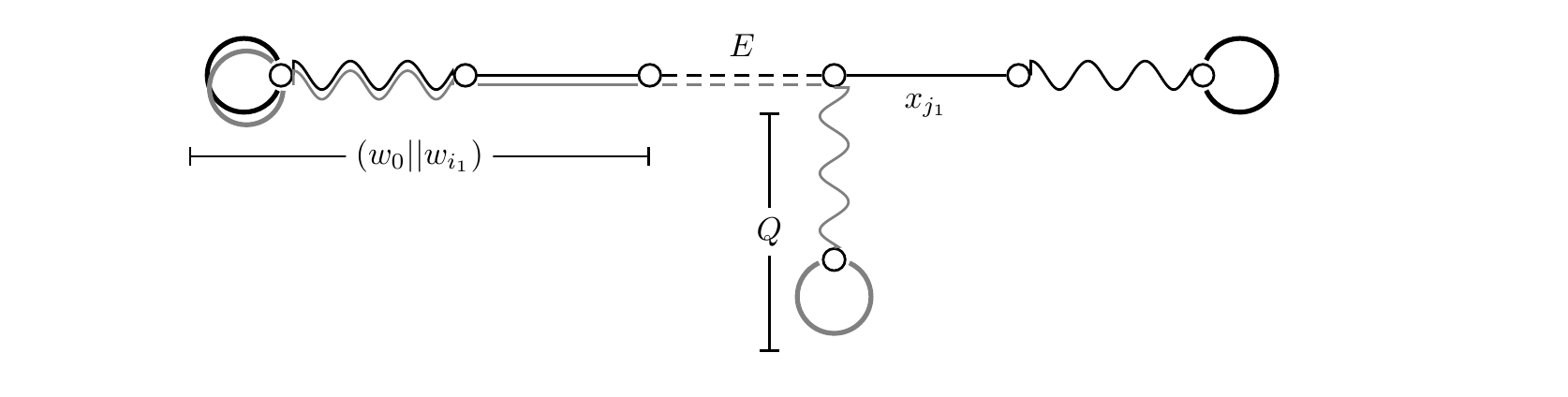}
\end{align*}
Now $x_{j_1}$ appears in $w_0$, so $u(B)_{j_1}=u(A)_{j_1}>0$, implying that there is a walk $w_{l_1}$ with non-zero coefficient appearing in $B$ with $u(w_{l_1})\neq 0$.  There are three cases:
\begin{itemize}
\item[(A)] $w_{j_1}$ is the (unique) walk appearing in $B$ with this property, then $x_{j_1}$ is an edge in $Q$;
\item[(B)] $w_{l_1}$ is not $w_{j_1}$, and is an alternating string;
\item[(C)] $w_{l_1}$ is an alternating band.
\end{itemize}
\begin{itemize}
\item[Case A:] This case impossible, for suppose that $w_{j_1}$ indeed contains $x_{j_1}$.  Said edge cannot be the first solid edge in $Q$, or else $x_{j_1}E(w_0||w_{i_1})$ would be a partial string common to both $w_0$ and $w_{i_1}$ with length one greater than $(w_0||w_{i_1})$, contradicting the definition.  Otherwise, $w_{i_1}$ takes one of the following two forms:
\begin{align*}
w_{i_1}&=\dotsc Ex_{j_1} \dotsc E (w_0||w_{i_1}) \\
w_{i_1}&=\dotsc x_{j_1} E C E (w_0||w_{i_1}),
\end{align*}
where $C$ is an alternating walk starting with the vertex $v'$ and ending with $v$.  In the former case, the walk $w_{i_1}$ could be written in the form $\dotsc E\dotsc x_{j_1}E(w_0||w_{i_1})$.  But $x_{j_1}E(w_0||w_{i_1})$ has greater length than $(w_0||w_{i_1})$.  Contradiction.  Finally, in the latter case, $w_{i_1}$ is not an irreducible walk since $EC$ is a band, so $w_{i_1} = \dotsc x_{j_1} E (w_0||w_{i_1})+ EC$, and the first summand is an alternating string with $l(w_0||\dotsc x_{j_1} E(w_0||w_{i_1}))>l(w_0||w_{i_1})$, contradicting the choice of $w_{i_1}$.  
\item[Case B:] Now we have $w_{l_1}$ an alternating string containing the edge $x_{j_1}$.  Let $Q'$ be the partial string in $w_{l_1}$ containing $x_{j_1}$ and not $E$, and $P'$ the partial string such that $Q'P' = w_{l_1}$.  This is depicted in the diagram below:
\begin{align*}
\includegraphics{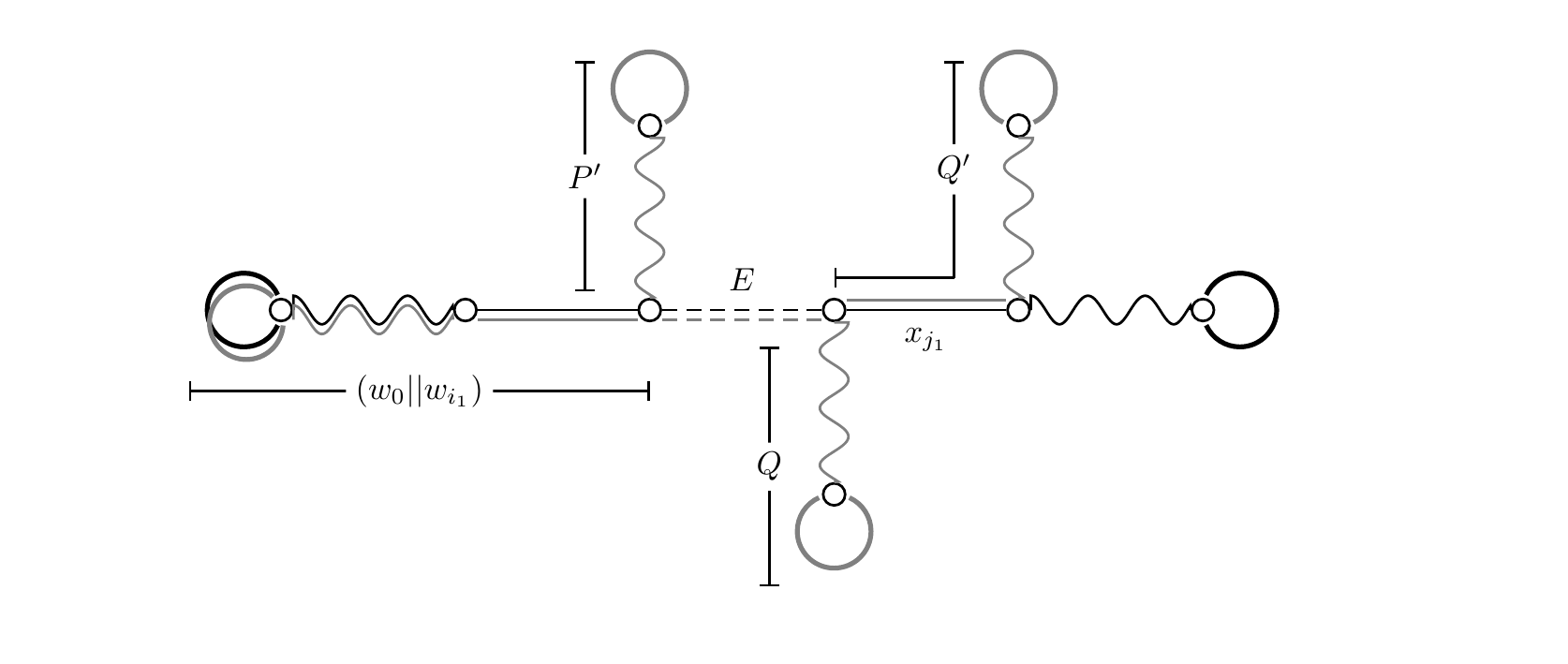}
\end{align*}
I.e., $Q'P' + Q(w_0||w_{i_1})$ appears in $B$.  Notice that this is an $X$-configuration about $E$, so $Q'P'+Q(w_0||w_{i_1}) \sim_X Q'(w_0||w_{i_1}) + QP'$.  Take $\Phi(B)= B-(Q'P'+Q(w_0||w_{i_1}))+(Q'(w_0||w_{i_1}) + QP')$.  Then $\Phi(B)\sim_X B$ and $\Phi(B)$ contains a summand, namely $Q'(w_0||w_{i_1})$, with $l(w_0|| (Q'(w_0||w_{i_1}))>l(w_0||w_{i_1})$ as claimed.
\item[Case C:] Finally, if $w_{l_1}=PEx_{j_1}$ is a band, then we are in the following situation:
\begin{align*}
\includegraphics{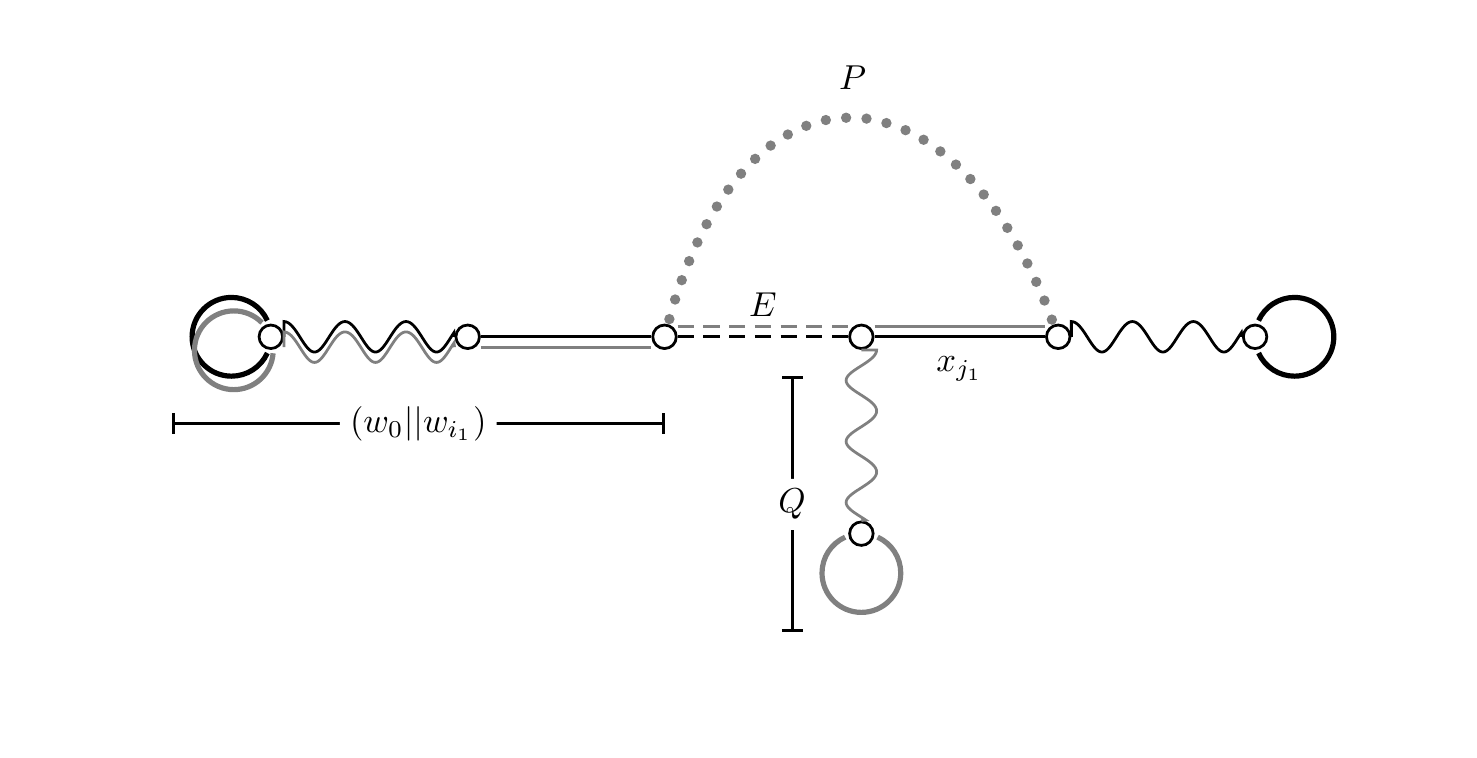}
\end{align*}
In this case, we can define $w_{i_2} = QEPx_{j_1}E(w_0||w_{i_1})$ (caution: this walk is not irreducible).  Then $l(w_0||w_{i_2})>l(w_0||w_{i_1})$, as desired.
\end{itemize}
\item[\underline{Case 2}:] Now suppose that $w_0$ is an alternating band.  Notice that we can assume (by symmetry) that there are no strings appearing as summands in $B$.  Again, for some band $w$ we will denote by $(w_0||w)$ any of the longest alternating paths contained in both $w_0$ and $w$. Let $y_1$ be some solid edge contained in $w_0$.  Since $u(w_0)_{y_1}\neq 0$, there must be a band $w_{i_1}$ appearing in $B$ passing through this edge.  This is depicted below, again the black edges form the band $w_0$ and the gray edges are from $w_{i_1}$.  
\begin{align*}
\includegraphics{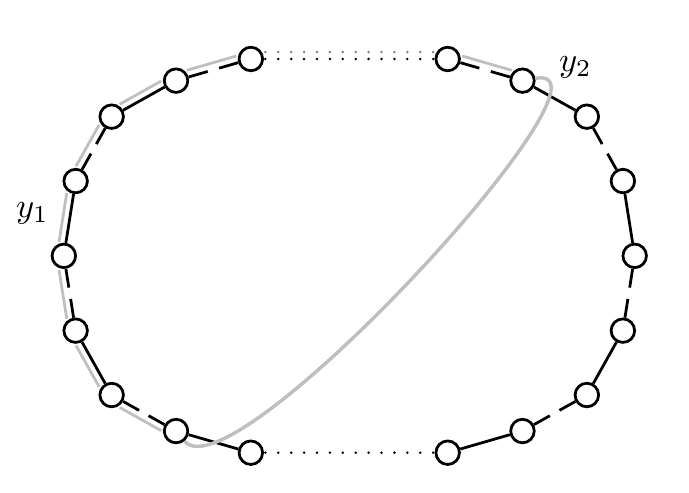}
\end{align*}
Fix an orientation on $w_0$, and suppose that $y_2$ is the first edge in $w_0$ (in the chosen orientation) which is not contained in $w_{i_1}$ as in the diagram.  But $u(B)_{y_2}\neq 0$, so there must be a band $w_{l_1}$ containing this edge.  By the same reasoning as the proof of case A for strings, if this band were $w_{i_1}$ (i.e., if $w_{i_1}$ contained $y_2$), then $w_{i_1}$ could be rewritten so as to contain a longer common subpath with $w_0$.  Therefore, this path is distinct from $w_{i_1}$.  There are two cases: 
\begin{itemize}
\item[Case A:] $w_{l_1}$ contains all other edges in $w_0$ as in the diagram including that labeled $y_2$:
\begin{align*}
\includegraphics{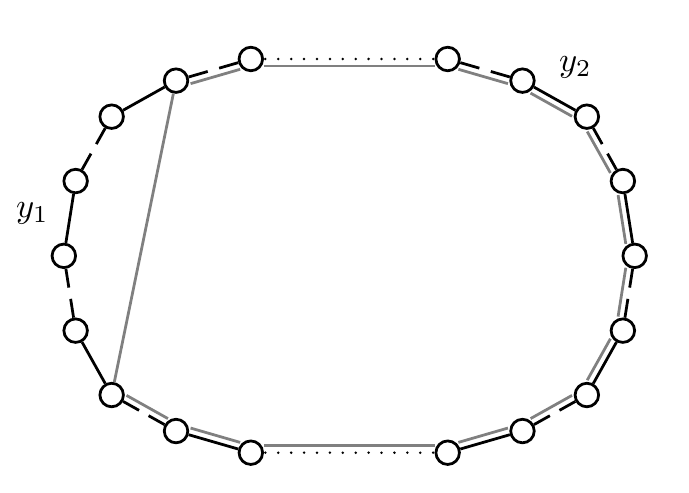}
\end{align*}
then $w_{i_1}$ and $w_{l_1}$ are in an $H$-configuration.  
\begin{align*}
\includegraphics{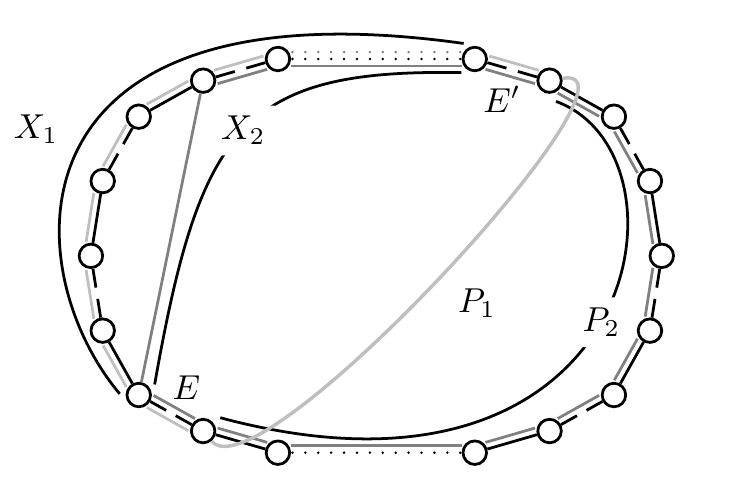}
\end{align*}
since $w_{i_1}=EX_1E'P_1$, and $w_{i_2}=EX_2E'P_2$.  Therefore 
\begin{align*}
w_{i_1}+w_{i_2}&=EX_1E'P_1 + EX_2E'P_2\\
	&= EX_1E'P_2 + EX_2E'P_1\\
	&= w_0 + EX_2E'P_1.\end{align*}
As such, $B \sim_X w_0 + B'$ with $u(B')=u(B)-u(w_0)<u(B)$.
\item[Case B:] $w_{l_1}$ does not contain all other edges in $w_0$:
\begin{align*}
\includegraphics{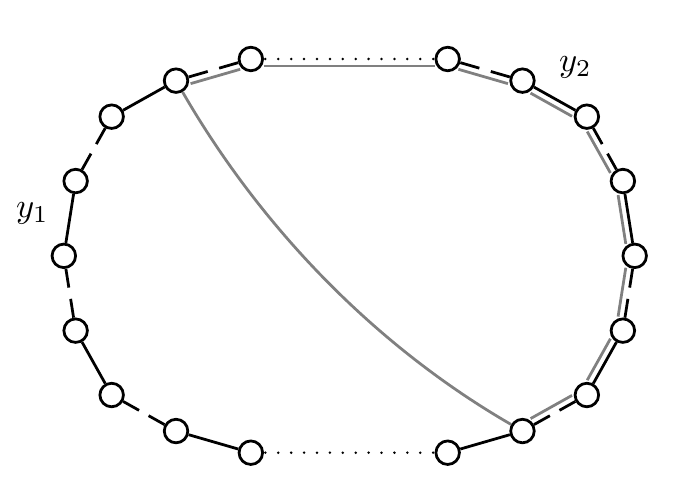} & \includegraphics{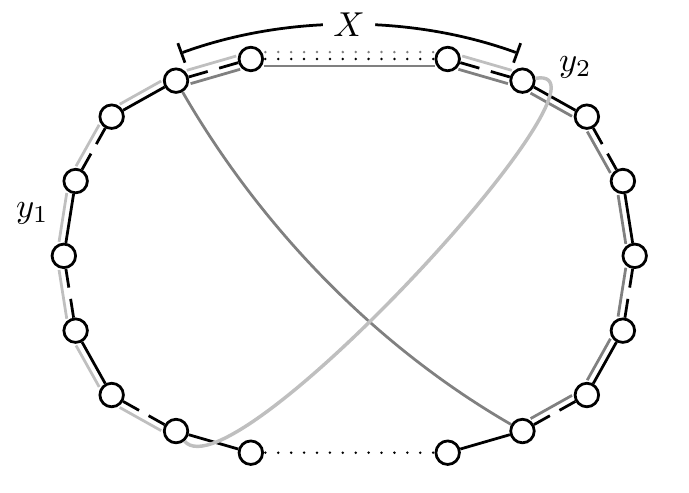}
\end{align*}
Let $X$ be the subpath common to both $w_{i_1}$ and $w_{i_2}$ as above, $P_1$ and $P_2$ the paths such that $w_{i_1}=P_1X$ and $w_{i_2}=P_2X$, respectively.  Then $w_{i_1}+w_{i_2}=XP_1XP_2$ is an alternating band (although clearly not irreducible).  Furthermore, $l(w_0||XP_1XP_2)>l(w_0||w_{i_1})$.  Since the length of $w_0$ is finite, iteration of this will introduce an $H$-configuration as in case A within $l(w_0)$ steps.  
\end{itemize}

\end{itemize}
\end{proof}

\end{proof}

\begin{example}
We conclude this section by describing the ring $\C[U(\underline{f})]$ for the semigroup in $\N^{10}$ given by the set of points $(a_1,\dotsc, a_5, b_1,\dotsc, b_5)$ satisfying the following equations:
\begin{align*}
a_1+a_2&=a_4+a_5\\
b_1+b_2&=b_4+b_5\\
a_2+a_3&=b_2+b_3\\
a_3+a_4&=b_3+b_4.
\end{align*}
The associated matching graph is
\begin{align*}
\xymatrix{
1\ar@[|(3)]@{-}@(l,u)^{a_1} \ar@{-}@[|(3)][dr]_{a_2} \ar@{--}@[|(3)][rrr] & & & 5\ar@[|(3)]@{-}@(r,u)_{a_5}\\
& 2\ar@{-}[r]^{a_3}\ar@[|(3)]@{--}[d] & 3\ar @{-}[ur]_{a_4} \ar@[|(3)]@{--}[d] & \\
& 7\ar@{-}[r]_{b_3}  & \ar@{-}[dr]^{b_4}8& \\
4\ar@{-}[ur]^{b_2} \ar@[|(4)]@{-}@(l,d)_{b_1} \ar@{--}@[|(3)][rrr] &&&9 \ar@[|(4)]@{-}@(r,d)^{b_5}}
\end{align*}
The alternating strings and bands are given below, with all but the labels of the solid edges suppressed:
\begin{align*}
X_1&:=a_1a_5 & X_2&:=b_1b_5\\
Y_1&:=a_1a_4b_4b_1& Y_2&:=a_5a_2b_2b_5\\
Z_1&:=a_2b_3a_4 & Z_2&:=b_2a_3b_4\\
B_1&:=a_2b_2b_4a_4 & B_2&:=a_3b_3.
\end{align*}
It can be easily verified that this is the complete list of the alternating walks.  The walks labeled by $X$ and $Y$ are strings, while those labeled by $Z$ and $B$ are bands.  The relation
\[
Y_1+Y_2 = B_2 + X_1+X_2
\]
can be viewed as an X-configuration about the solid edge containing the vertices $1$ and $5$, with $P_1=a_1$, $p_2=b_5b_2a_2$, $Q_1=a_5$ and $Q_2=b_1b_4a_4$.  The relation
\[
Z_1+Z_2=B_1+B_2
\]
can be seen as an H-configuration about the edges containing $2, 7$ and $3,8$.  It can be shown that these are the only relations, so 
\[ 
\C[U(\underline{f})] = \C[X_1, X_2, Y_1, Y_2, Z_1, Z_2, B_1, B_2]/(Y_1\cdot Y_2 -B_2\cdot X_1\cdot X_2, Z_1\cdot Z_2 - B_1\cdot B_2).\]

\end{example}

\section{Degree Bounds}\label{sec:corollary}
It is a simple consequence of section \ref{sec:TheCoordinateRing} that for $\lambda \in \Lambda_{SI}(Q, c, \beta, r)$, the function $m_\lambda$ is of degree
\[
	\sum\limits_{a\in Q_1} \lvert \lambda(a) \rvert
	\]
under the usual grading on the polynomial ring.  We will use this and the map $(u, y) \mapsto \lambda_{u, y}$ to give degree bounds on the generators and relations for $\SI_{Q, c}(\beta, r)$.  Recall that there is a second grading on $\SI_{Q, c}(\beta, r)$, as in proposition \ref{prop:KGrading}, given by the connected components of the partition equivalence graph.  The first corollary relates to this grading.
\begin{corollary}
The generators for $\SI_{Q, C}(\beta, r)$ occur in multi-degrees bounded by $\varphi_{u_\lambda}(e)\leq 2$ and $y_\lambda(e)\leq 1$.  
\end{corollary}

As for degree bounds in the polynomial ring, we have the following:

\begin{corollary}
The generators for $\SI_{Q, c}(\beta, r)$ occur in total degrees bounded by $$2 \sum\limits_{a\in Q_0} \binom{r(a)+1}{2}.$$
\end{corollary}
\begin{proof}
Since $\lambda(a)_{r(a)}\leq 2$, and $\lambda(a)_{i+1}\leq\lambda(a)_i\leq \lambda(a)_{i+1}+2$, we have 
\begin{align*}
\operatorname{deg}(m_\lambda) = \sum\limits_{a\in Q_1} \lvert \lambda(a) \rvert &\leq \sum\limits_{a\in Q_1} \sum\limits_{i=1}^{r(a)} 2i \\ 
	&= 2\sum\limits_{a\in Q_1} \binom{r(a)+1}{2}.
	\end{align*}
	\end{proof}
\begin{corollary}
The relations for $\SI_{Q, c}(\beta, r)$ occur in total degrees bounded by $$8\sum\limits_{a\in Q_1} \binom{r(a)+1}{2}.$$
\end{corollary}
\begin{proof}
We may assume that in an $X$-relation, none of the arms contains a subband, so by theorem \ref{thm:degreesofgenerators}, we have that for each arm $u(a)\leq 2$ and $\varphi_u(e)\leq 2$ for any $e$.  Therefore, on $P_1P_2 \cdot Q_1Q_2$, the bounds become $u(a)\leq 8$ and $\varphi_u(e)\leq 8$.  The bound is derived similarly to the previous corollary.  The same technique works for $H$-relations as well, so the bound is as desired.\end{proof} 


\end{document}